\documentclass{article}
\usepackage{amsmath}
\usepackage{amssymb}

\newtheorem{thm}{Theorem}[section]

\newtheorem{lem}[thm]{Lemma}

\newtheorem{remark}[thm]{Remark}
\numberwithin{equation}{section}
\newenvironment{proof}[1][Proof]{\noindent\textbf{#1.} }{\ \rule{0.5em}{0.5em}}

\begin{document}

\title{Linear Inviscid Damping for Couette Flow in Stratified Fluid}
\author{Jincheng Yang \\
School of Mathematics and Statistics\\
Xi'an Jiaotong University\\
Xi'an, 710049, China \and Zhiwu Lin \\
School of Mathematics\\
Georgia Institute of Technology\\
Atlanta, GA 30332-0160, USA}
\date{}
\maketitle

\begin{abstract}
We study the inviscid damping of Couette flow with an exponentially
stratified density. The optimal decay rates of the velocity field and the
density are obtained for general perturbations with minimal regularity. For
Boussinesq approximation model, the decay rates we get are consistent with
the previous results in the literature. We also study the decay rates for
the full Euler equations of stratified fluids, which were not studied
before. For both models, the decay rates depend on the Richardson number in
a very similar way. Besides, we also study the dispersive decay due to the
exponential stratification when there is no shear.
\end{abstract}

\section{Introduction}

Couette flow in exponentially stratified fluid is a shear flow $U(y)=Ry$
with the density profile $\rho _{0}(y)=Ae^{-\beta y}$. The stability of such
a flow was first studied by Taylor (\cite{taylor31}) in the half space by
the method of normal modes. He presented a convincing but somewhat incomplete
analysis to show that the spectrum of the linearized equation (now called
Taylor-Goldstein equation) is quite different when the Richardson number $%
B^{2}=\frac{\beta g}{R^{2}}$ ($g$ is the gravitational constant) is greater or
less than $1/4.$ He found that there exist infinitely many discrete neutral
eigenvalues when $B^{2}>\frac{1}{4}$ and no such neutral eigenvalues exist
when $B^{2}<\frac{1}{4}$. This claim was later proved by Dyson (\cite{dyson}%
) and Dikii (\cite{dikii-roots}). However, Taylor did not provide a clear
answer to the problem of stability of Couette flow. From 1950s, there have
been lots of work trying to understand the stability of stratified Couette
flow, by studying the initial value problem. They include H\o iland (\cite%
{hoiland53}), Eliassen et al. (\cite{eliassen}), Case (\cite{case-Euler-decay}), Dikii (\cite{dikii60}), Kuo (\cite{kuo}), Hartman (\cite{hartman74}), Chimonas (%
\cite{chimonas79}), Brown and Stewartson (\cite{brown-stewartson80}),
Farrell and Ioannou (\cite{farrel93}). We refer to Section 3.2.3 of the book
of Yaglom (\cite{Yaglom12}) for a detailed survey of the literature. Most of
the papers used the Boussinesq approximation. One exception is Dikii (\cite%
{dikii60}), where he proved the Liapunov stability of Couette flow in the
half space for the full stratified Euler equations, and for any $B^{2}>0$.
We note that for the exponentially stratified fluid (i.e. $\rho
_{0}(y)=Ae^{-\beta y}$), the Boussinesq approximation is valid only when $%
\beta $ is small. One interesting result following from the initial value
approach is the inviscid damping of velocity fields. Such inviscid damping
phenomena was known by Orr (\cite{orr}) in 1907, where the Couette flow in a
homogeneous fluid was considered. Orr showed that the horizontal and
vertical velocities decay by $t^{-1}$ and $t^{-2}$ respectively. Such
damping is not due to the viscosity, but instead is due to the mixing of the
vorticity under the Couette flow. In recent years, the inviscid damping
phenomena attracted new attention. In \cite{Lin-zeng}, Lin and Zeng showed
that if we consider initial (vorticity) perturbation in the Sobolev space $%
H^{s}$ $\left( s<\frac{3}{2}\right) $ then the nonlinear damping is not true
due to the existence of nonparallel steady flows of the form of Kelvin's
cats eye near Couette. In \cite{bedrossian-masmoudi}, Bedrossian and
Masmoudi proved the nonlinear inviscid damping for perturbations near
Couette in Gevrey class (i.e. almost analytic). The linear inviscid damping
for more general shear flows in a homogeneous fluid were also studied in
\cite{zillinger-arma} \cite{wei-zhang-zhao}.

In this paper, our goal is to get the precise estimates of linear decay
rates for Couette flow in exponentially stratified fluid, which might be
useful in the future study of nonlinear damping. We restrict ourselves to
the case in the whole space. The including of the boundary (half space,
finite channel) causes additional complication, as can be seen from Taylor's
results mentioned at the beginning.

Our first result is about the linear decay estimates for solutions of the
linearized equations under Boussinesq approximation. Consider the steady
shear flow $\boldsymbol{v}_{0}=(Ry,0)\ $with an exponentially stratified
density profile $\rho _{0}(y)=Ae^{-\beta y}$, where $R\in \mathbb{R},A >
0,\beta \geq 0$ are constants. Denote $B^{2}=\frac{\beta g}{R^{2}}$ to be
the Richardson number. When $\beta $ is small, we approximate $\rho _{0}(y)$
by $A\left( 1-\beta y\right) $ and the linearized equations under the
Boussinesq approximation (see Section 2.1) is
\begin{equation}
\left( \partial _{t}+Ry\partial _{x}\right) \Delta \psi =-\partial
_{x}\left( \frac{\rho }{A}\right) g,  \label{eqn-stream-Bou}
\end{equation}%
\begin{equation}
\ \left( \partial _{t}+Ry\partial _{x}\right) \left( \frac{\rho }{A}\right)
=\beta \partial _{x}\psi ,  \label{eqn-density-bou}
\end{equation}%
where $\psi $ and $\frac{\rho }{A}$ are the perturbations of stream function
and relative density variation.

\begin{thm}
\label{thm-boussinesq} Let $\left( \psi (t;x,y),\frac{\rho }{A}%
(t;x,y)\right) $ be the solution of (\ref{eqn-stream-Bou})-(\ref%
{eqn-density-bou}) with the initial data
\begin{equation*}
\psi (0;x,y)=\psi ^{0}(x,y),\ \ \frac{\rho (0;x,y)}{A}=\rho ^{0}(x,y),
\end{equation*}%
where $y\in \mathbb{R}$ and $x$ is periodic with period $L$. Denote the
velocity $\boldsymbol{v}=\nabla ^{\perp }\psi =\left( v^{x},v^{y}\right) $.
Below, $f\lesssim g$ stands for $f\leq Cg$ for a constant C depending only
on $R,\beta ,g.\ $We denote $\left\langle f\right\rangle :=\sqrt{1+f^{2}}$
and $P_{\neq 0}$ to be the projection to nonzero Fourier modes (in $x$),
that is,
\begin{equation*}
P_{\neq 0}f(t;x,y)=f(t;x,y)-\frac{1}{L}\int_{0}^{L}f(t;x,y)dx.
\end{equation*}%
The following estimates hold true:

(i) If $0<B^{2}<\frac{1}{4}$, let $\nu =\sqrt{\frac{1}{4}-B^{2}}$,$\ $then
\begin{align*}
\Vert P_{\neq 0}v^{x}\Vert _{L^{2}}& \lesssim \left\langle t\right\rangle ^{-%
\frac{1}{2}+\nu }\left( \Vert \psi ^{0}\Vert _{H_{x}^{1}H_{y}^{2}}+\Vert
\rho ^{0}\Vert _{L_{x}^{2}H_{y}^{1}}\right) , \\
\Vert v^{y}\Vert _{L^{2}}& \lesssim \left\langle t\right\rangle ^{-\frac{3}{2%
}+\nu }\left( \Vert \psi ^{0}\Vert _{H_{x}^{1}H_{y}^{3}}+\Vert \rho
^{0}\Vert _{L_{x}^{2}H_{y}^{2}}\right) , \\
\Vert P_{\neq 0}\frac{\rho }{A}\Vert _{L^{2}}& \lesssim \left\langle
t\right\rangle ^{-\frac{1}{2}+\nu }\left( \Vert \psi ^{0}\Vert
_{H_{x}^{1}H_{y}^{2}}+\Vert \rho ^{0}\Vert _{L_{x}^{2}H_{y}^{1}}\right) .
\end{align*}

(ii) If $B^{2}>\frac{1}{4}$ then
\begin{align*}
\Vert P_{\neq 0}v^{x}\Vert _{L^{2}}& \lesssim \left\langle t\right\rangle ^{-%
\frac{1}{2}}\left( \Vert \psi ^{0}\Vert _{H_{x}^{1}H_{y}^{2}}+\Vert \rho
^{0}\Vert _{L_{x}^{2}H_{y}^{1}}\right) , \\
\Vert v^{y}\Vert _{L^{2}}& \lesssim \left\langle t\right\rangle ^{-\frac{3}{2%
}}\left( \Vert \psi ^{0}\Vert _{H_{x}^{1}H_{y}^{3}}+\Vert \rho ^{0}\Vert
_{L_{x}^{2}H_{y}^{2}}\right) , \\
\Vert P_{\neq 0}\frac{\rho }{A}\Vert _{L^{2}}& \lesssim \left\langle
t\right\rangle ^{-\frac{1}{2}}\left( \Vert \psi ^{0}\Vert
_{H_{x}^{1}H_{y}^{2}}+\Vert \rho ^{0}\Vert _{L_{x}^{2}H_{y}^{1}}\right) .
\end{align*}

(iii) If $B^{2}=\frac{1}{4}$, then%
\begin{align*}
\Vert P_{\neq 0}v^{x}\Vert _{L^{2}}& \lesssim \left\langle t\right\rangle ^{-%
\frac{1}{2}}\left\langle \log \left\langle t\right\rangle \right\rangle
\left( \Vert \psi ^{0}\Vert _{H_{x}^{1}H_{y}^{2}}+\Vert \rho ^{0}\Vert
_{L_{x}^{2}H_{y}^{1}}\right) , \\
\Vert v^{y}\Vert _{L^{2}}& \lesssim \left\langle t\right\rangle ^{-\frac{3}{2%
}}\left\langle \log \left\langle t\right\rangle \right\rangle \left( \Vert
\psi ^{0}\Vert _{H_{x}^{1}H_{y}^{3}}+\Vert \rho ^{0}\Vert
_{L_{x}^{2}H_{y}^{2}}\right) , \\
\Vert P_{\neq 0}\frac{\rho }{A}\Vert _{L^{2}}& \lesssim \left\langle
t\right\rangle ^{-\frac{1}{2}}\left\langle \log \left\langle t\right\rangle
\right\rangle \left( \Vert \psi ^{0}\Vert _{H_{x}^{1}H_{y}^{2}}+\Vert \rho
^{0}\Vert _{L_{x}^{2}H_{y}^{1}}\right) .
\end{align*}

(iv) If $B^{2}=0$, i.e., $\beta =0$, then $\left\Vert \frac{\rho }{A}%
\right\Vert _{L^{2}}\left( t\right) =\Vert \rho ^{0}\Vert _{L^{2}}$ and
\begin{align*}
\Vert P_{\neq 0}v^{x}\Vert _{L^{2}}& \lesssim \Vert \rho ^{0}\Vert
_{L_{x}^{2}H_{y}^{1}}+\left\langle t\right\rangle ^{-1}\Vert \psi ^{0}\Vert
_{H_{x}^{1}H_{y}^{3}}, \\
\Vert v^{y}\Vert _{L^{2}}& \lesssim \left\langle t\right\rangle ^{-1}\Vert
\rho ^{0}\Vert _{L_{x}^{2}H_{y}^{2}}+\left\langle t\right\rangle ^{-2}\Vert
\psi ^{0}\Vert _{H_{x}^{1}H_{y}^{4}}.
\end{align*}

(v) If $B^{2}=\infty $, i.e. $R=0$, then $\frac{g}{\beta }\left\Vert \frac{%
\rho }{A}\right\Vert _{L^{2}}^{2}+\left\Vert \boldsymbol{v}\right\Vert
_{L^{2}}^{2}\ $is conserved. The following decay estimates hold true in $%
L_{x}^{2}L_{y}^{\infty }$,
\begin{align*}
\Vert P_{\neq 0}v^{x}\Vert _{L_{x}^{2}L_{y}^{\infty }}& \lesssim |t|^{-\frac{%
1}{3}}\left( \Vert \psi ^{0}\Vert _{H_{x}^{3/2}\left( H_{y}^{9/2}\cap
W_{y}^{1,1}\right) }+\Vert \rho ^{0}\Vert _{H_{x}^{1/2}\left(
H_{y}^{9/2}\cap W_{y}^{1,1}\right) }\right) , \\
\Vert v^{y}\Vert _{L_{x}^{2}L_{y}^{\infty }}& \lesssim |t|^{-\frac{1}{3}%
}\left( \Vert \psi ^{0}\Vert _{H_{x}^{5/2}\left( H_{y}^{7/2}\cap
L_{y}^{1}\right) }+\Vert \rho ^{0}\Vert _{H_{x}^{3/2}\left( H_{y}^{7/2}\cap
L_{y}^{1}\right) }\right) , \\
\Vert P_{\neq 0}\frac{\rho }{A}\Vert _{L_{x}^{2}L_{y}^{\infty }}& \lesssim
|t|^{-\frac{1}{3}}\left( \Vert \psi ^{0}\Vert _{H_{x}^{5/2}\left(
H_{y}^{9/2}\cap W_{y}^{1,1}\right) }+\Vert \rho ^{0}\Vert
_{H_{x}^{3/2}\left( H_{y}^{7/2}\cap L_{y}^{1}\right) }\right) .
\end{align*}
\end{thm}

Theorem \ref{thm-boussinesq} gives a complete picture of the linear damping
for the Couette flow in an exponentially stratified fluid in an infinite
channel (i.e. $-\infty <y<+\infty $ and $x$ periodic). More specifically, we
obtain optimal decay rates for initial perturbations of minimal regularity.
We make some comments to relate our results to the previous works on this
problem. When $B^{2}>\frac{1}{4}$, the decay rates $t^{-\frac{3}{2}}$ for $%
v^{y}$ and $t^{-\frac{1}{2}}$ for $v^{x}$ were obtained by Booker and
Bretherton (\cite{booker67}) for a special class of solutions, which
generalized the earlier results in \cite[Chap. 5]{phillips} for $B^{2}\gg 1$%
. In \cite{hartman74}, the decay rates as in Theorem \ref{thm-boussinesq}
(i)-(iii) were obtained for special solutions by hypergeometric functions,
which are similar to $g_{1},g_{2}$ defined in (\ref{g_1}) and (\ref{g_2}).
However, it was not shown in \cite{hartman74} that general solutions can be
expressed by these special solutions. Chimonas (\cite{chimonas79})
considered the case $B^{2}<\frac{1}{4}$ and wrongly claimed that $v^{y}$
decays at the rate $t^{2\nu -1}$ and $v^{x}$ grows by $t^{2\nu }$. Later, an
error in \cite{chimonas79} was pointed out by Brown and Stewartson (\cite%
{brown-stewartson80}), where they also found the correct decay rates as in
Theorem \ref{thm-boussinesq}. In \cite{brown-stewartson80}, the initial
value problem was solved for analytic initial data by taking the Laplace
transform in time and then the decay rates were obtained from the asymptotic
analysis of the inverse Laplace transform of the solutions. Moreover, it was
assumed in \cite{brown-stewartson80} that the discrete neutral eigenvalues
do no exist, such that there are no poles in the Laplace transform of their
solutions. In our analysis, we do not need to assume the nonexistence of
discrete neutral eigenvalues, which actually follows as a corollary from the
decay estimates in Theorem \ref{thm-boussinesq} for any $B^{2}>0$. This
contrasts significantly with the case in the half space (\cite{taylor31}
\cite{dikii-roots} \cite{dyson}) or in a finite channel (\cite{eliassen}),
where it was shown that there exist infinitely many discrete neutral
eigenvalues when $B^{2}>\frac{1}{4}$. In Theorem \ref{thm-boussinesq}, the
decay rates are optimal with the minimal regularity requirement for the
initial data. In particular, when $B^{2}<\infty \ $it suffices to have the
initial perturbations of vorticity and density variation $\omega \left(
0\right) ,\rho ^{0}\in H^{1}\ $to get the optimal decay for $\left\Vert
v^{x}\right\Vert _{L^{2}}$, and $\omega \left( 0\right) ,\rho ^{0}\in H^{2}$
to get the optimal decay for $\left\Vert v^{y}\right\Vert _{L^{2}}$. These
minimal regularity requirement on the initial data are consistent with the
results in \cite{Lin-zeng} for the Couette flow with constant density.
Moreover, if $B\rightarrow 0+$ (i.e. $\nu \rightarrow \frac{1}{2}-$), the
decay rates for the horizontal and vertical velocities are $t^{-\frac{1}{2}%
+\nu }$ and $t^{-\frac{3}{2}+\nu}$ respectively even when $\rho ^{0}=0$, which
are almost one order slower than the rates ($t^{-1}$ and $t^{-2}$
respectively) for homogeneous fluids (i.e. $B=0$). This suggests that the
stratified effects cannot be ignored even when the steady density is a small
deviation of the constant.

The decay rate $t^{-\frac{1}{3}\ }$for the case $B^{2}=\infty $ (i.e. no
shear flow) is optimal (see the example at the end of Section 6.1). When $%
\left( x,y\right) \in \mathbb{R}^{2}$, the optimal decay rate was shown to
be $t^{-\frac{1}{2}}$ in \cite{elgindi-sima}. We note that the decay
mechanisms are very different for the cases of $B^{2}=\infty $ and $%
B^{2}<\infty $. When $B^{2}<\infty $, the decay of $\left\Vert \boldsymbol{v}%
\right\Vert _{L^{2}}$ is due to the mixing of vorticity caused by the shear
motion. When $B^{2}=\infty $, $\left\Vert \boldsymbol{v}\right\Vert _{L^{2}}$
does not decay while the decay of $\left\Vert \boldsymbol{v}\right\Vert
_{L^{\infty }}$ is due to dispersive effects of the linear waves in a stably
stratified fluid.

Most papers on Couette flow used the Boussinesq approximation to analyze the
linearized solutions. However, this approximation is valid only when $\beta $
is small. For $\beta $ not small, the full Euler equations should be used.
In this case, the linearized equatuions at the Couette flow $\left(
Ry,0\right) $ with the exponential density profile $\rho _{0}(y)=Ae^{-\beta
y}$ become
\begin{equation}
\beta \left[ R\partial _{x}-\left( \partial _{t}+Ry\partial _{x}\right)
\partial _{y}\right] \psi +\left( \partial _{t}+Ry\partial _{x}\right)
\Delta \psi =-\partial _{x}\left( \frac{\rho }{\rho _{0}}\right) g,
\label{eqn-stream-full}
\end{equation}%
\begin{equation}
\left( \partial _{t}+Ry\partial _{x}\right) \left( \frac{\rho }{\rho _{0}}%
\right) =\beta \partial _{x}\psi .  \label{eqn-density-full}
\end{equation}%
We obtain similar results on decay estimates in the $e^{-\frac{1}{2}\beta y}$
weighted norms.

\begin{thm}
\label{thm-full-euler}Let $\left( \psi (t;x,y),\frac{\rho }{\rho _{0}}%
(t;x,y)\right) $ be the solution of (\ref{eqn-stream-full})-(\ref%
{eqn-density-full}) with the initial data
\begin{equation*}
\psi (0;x,y)=\psi ^{0}(x,y),\ \ \frac{\rho (0;x,y)}{\rho _{0}(y)}=\rho
^{0}(x,y),
\end{equation*}%
where $y\in \mathbb{R}$ and $x$ is periodic with period $L$. Let $%
\boldsymbol{v}=\nabla ^{\perp }\psi =\left( v^{x},v^{y}\right) $. The
following is true: 

(i) If $0<B^{2}<\frac{1}{4}$, let $\nu =\sqrt{\frac{1}{4}%
-B^{2}}$, then
\begin{align*}
\Vert e^{-\frac{1}{2}\beta y}P_{\neq 0}v^{x}\Vert _{L^{2}}& \lesssim
\left\langle t\right\rangle ^{-\frac{1}{2}+\nu }\left( \Vert e^{-\frac{1}{2}%
\beta y}\psi ^{0}\Vert _{H_{x}^{1}H_{y}^{2}}+\Vert e^{-\frac{1}{2}\beta
y}\rho ^{0}\Vert _{L_{x}^{2}H_{y}^{1}}\right) , \\
\Vert e^{-\frac{1}{2}\beta y}v^{y}\Vert _{L^{2}}& \lesssim \left\langle
t\right\rangle ^{-\frac{3}{2}+\nu }\left( \Vert e^{-\frac{1}{2}\beta y}\psi
^{0}\Vert _{H_{x}^{1}H_{y}^{3}}+\Vert e^{-\frac{1}{2}\beta y}\rho ^{0}\Vert
_{L_{x}^{2}H_{y}^{2}}\right) , \\
\Vert e^{-\frac{1}{2}\beta y}P_{\neq 0}\rho /\rho _{0}\Vert _{L^{2}}&
\lesssim \left\langle t\right\rangle ^{-\frac{1}{2}+\nu }\left( \Vert e^{-%
\frac{1}{2}\beta y}\psi ^{0}\Vert _{H_{x}^{1}H_{y}^{2}}+\Vert e^{-\frac{1}{2}%
\beta y}\rho ^{0}\Vert _{L_{x}^{2}H_{y}^{1}}\right) .
\end{align*}

(ii) If $B^{2}>\frac{1}{4}$ then
\begin{align*}
\Vert e^{-\frac{1}{2}\beta y}P_{\neq 0}v^{x}\Vert _{L^{2}} &\lesssim
\left\langle t\right\rangle ^{-\frac{1}{2}}\left( \Vert e^{-\frac{1}{2}\beta
y}\psi ^{0}\Vert _{H_{x}^{1}H_{y}^{2}}+\Vert e^{-\frac{1}{2}\beta y}\rho
^{0}\Vert _{L_{x}^{2}H_{y}^{1}}\right) , \\
\Vert e^{-\frac{1}{2}\beta y}v^{y}\Vert _{L^{2}} &\lesssim \left\langle
t\right\rangle ^{-\frac{3}{2}}\left( \Vert e^{-\frac{1}{2}\beta y}\psi
^{0}\Vert _{H_{x}^{1}H_{y}^{3}}+\Vert e^{-\frac{1}{2}\beta y}\rho ^{0}\Vert
_{L_{x}^{2}H_{y}^{2}}\right) , \\
\Vert e^{-\frac{1}{2}\beta y}P_{\neq 0}\rho /\rho _{0}\Vert _{L^{2}}
&\lesssim \left\langle t\right\rangle ^{-\frac{1}{2}}\left( \Vert e^{-\frac{1%
}{2}\beta y}\psi ^{0}\Vert _{H_{x}^{1}H_{y}^{2}}+\Vert e^{-\frac{1}{2}\beta
y}\rho ^{0}\Vert _{L_{x}^{2}H_{y}^{1}}\right) .
\end{align*}

(iii) If $B^{2}=\frac{1}{4}$, then
\begin{align*}
\Vert e^{-\frac{1}{2}\beta y}P_{\neq 0}v^{x}\Vert _{L^{2}} &\lesssim
\left\langle t\right\rangle ^{-\frac{1}{2}}\left\langle \log \left\langle
t\right\rangle \right\rangle \left( \Vert e^{-\frac{1}{2}\beta y}\psi
^{0}\Vert _{H_{x}^{1}H_{y}^{2}}+\Vert e^{-\frac{1}{2}\beta y}\rho ^{0}\Vert
_{L_{x}^{2}H_{y}^{1}}\right) , \\
\Vert e^{-\frac{1}{2}\beta y}v^{y}\Vert _{L^{2}} &\lesssim \left\langle
t\right\rangle ^{-\frac{3}{2}}\left\langle \log \left\langle t\right\rangle
\right\rangle \left( \Vert e^{-\frac{1}{2}\beta y}\psi ^{0}\Vert
_{H_{x}^{1}H_{y}^{3}}+\Vert e^{-\frac{1}{2}\beta y}\rho ^{0}\Vert
_{L_{x}^{2}H_{y}^{2}}\right) , \\
\Vert e^{-\frac{1}{2}\beta y}P_{\neq 0}\rho /\rho _{0}\Vert _{L^{2}}
&\lesssim \left\langle t\right\rangle ^{-\frac{1}{2}}\left\langle \log
\left\langle t\right\rangle \right\rangle \left( \Vert e^{-\frac{1}{2}\beta
y}\psi ^{0}\Vert _{H_{x}^{1}H_{y}^{2}}+\Vert e^{-\frac{1}{2}\beta y}\rho
^{0}\Vert _{L_{x}^{2}H_{y}^{1}}\right) .
\end{align*}

(iv) If $B^{2}=0$, i.e, $\beta =0$, then the results are the same as in the
Boussinesq case, with $\rho /\rho _{0}$ replacing $\frac{\rho }{A}$.

(v) If $B^{2}=\infty $, i.e. $R=0$, then
\begin{equation*}
\left\Vert e^{-\frac{1}{2}\beta y} \boldsymbol{v} \right\Vert _{L^{2}}^{2}+%
\frac{g}{\beta }\left\Vert e^{-\frac{1}{2}\beta y}\frac{\rho }{\rho _{0}}%
\right\Vert _{L^{2}}^{2}
\end{equation*}%
is conserved and
\begin{align*}
\Vert e^{-\frac{1}{2}\beta y}P_{\neq 0}v^{x}\Vert _{L_{x}^{2}L_{y}^{\infty
}} \lesssim& |t|^{-\frac{1}{3}}\left( \Vert e^{-\frac{1}{2}\beta y}\psi
^{0}\Vert _{H_{x}^{3/2}\left( H_{y}^{9/2}\cap W_{y}^{1,1}\right) } \right. \\
&\left. +\Vert e^{-\frac{1}{2}\beta y}\rho ^{0}\Vert _{H_{x}^{1/2}\left(
H_{y}^{9/2}\cap W_{y}^{1,1}\right) }\right) , \\
\Vert e^{-\frac{1}{2}\beta y}v^{y}\Vert _{L_{x}^{2}L_{y}^{\infty }}
\lesssim& |t|^{-\frac{1}{3}}\left( \Vert e^{-\frac{1}{2}\beta y}\psi
^{0}\Vert _{H_{x}^{5/2}\left( H_{y}^{7/2}\cap L_{y}^{1}\right) } \right. \\
&\left. +\Vert e^{-\frac{1}{2}\beta y}\rho ^{0}\Vert _{H_{x}^{3/2}\left(
H_{y}^{7/2}\cap L_{y}^{1}\right) }\right) , \\
\Vert e^{-\frac{1}{2}\beta y}P_{\neq 0}\rho /\rho _{0}\Vert
_{L_{x}^{2}L_{y}^{\infty }} \lesssim& |t|^{-\frac{1}{3}}\left( \Vert e^{-%
\frac{1}{2}\beta y}\psi ^{0}\Vert _{H_{x}^{5/2}\left( H_{y}^{9/2}\cap
W_{y}^{1,1}\right) } \right. \\
&\left. +\Vert e^{-\frac{1}{2}\beta y}\rho ^{0}\Vert _{H_{x}^{3/2}\left(
H_{y}^{7/2}\cap L_{y}^{1}\right) }\right) .
\end{align*}
\end{thm}

Compared with Theorem \ref{thm-boussinesq}, it is interesting to note that
for the $e^{-\frac{1}{2}\beta y}$ weighted $v$ and $\rho $, the decay rates
and the initial regularity requirement for the full equations are exactly
the same as in the Boussinesq approximation.

Lastly, we make some comments on the proof. First, we use Fourier transform
on the linearized equations in the sheared coordinates and then reduce them
to a second order ODE for the stream function. The general solution is
expressed by two special solutions of hypergeometric functions. The decay
rates and initial regularity are then obtained by using the asymptotic
behaviors of hypergeometric functions. In the case of $B^{2}=\infty $ (i.e.
no shear), the decay rates are obtained by the dispersive estimates and
oscillatory integrals.

This paper is organized as follows. In Section 2, we derive the linearized
equations and give some identities of hypergeometric functions to be used
later. In Section 3, we solve the linearized equations by hypergeometric
functions. In Section 4 and 5, we obtain the decay estimates from the
solution formula for the case $B^{2}<\infty $. In Section 6, the dispersive
decay estimates are obtained for the case $B^{2}=\infty $.

\section{Preliminary}

\subsection{Linearized Euler Equation and Boussinesq Approximation}

The equations for two dimensional inviscid incompressible flows in
stratified fluids are%
\begin{equation}
\rho \left( \partial _{t}+\boldsymbol{v}\cdot \nabla \right) \boldsymbol{v}%
+\nabla p=\rho \boldsymbol{g},  \label{eqn-momentum-Euler}
\end{equation}%
\begin{equation}
\left( \partial _{t}+\boldsymbol{v}\cdot \nabla \right) \rho =0,
\label{eqn-density-Euler}
\end{equation}%
\begin{equation*}
\nabla \cdot \boldsymbol{v}=0,
\end{equation*}%
where $(x,y)\in \mathbb{T}\times \mathbb{R}$,$\ \boldsymbol{v=}\left(
v^{x},v^{y}\right) $ is the velocity, $\rho \ $is the density and $%
\boldsymbol{g}=(0,-g)$ is the gravitational acceleration directing downward
with $g$ being the gravitational constant. The simplest stationary solution
is the shear flow, with $\boldsymbol{v}_{0}=(U(y),0)$ and $\rho _{0}=\rho
_{0}(y)$. Let $\psi =\psi (t;x,y)$ be the stream function such that $%
\boldsymbol{v}=\nabla ^{\perp }\psi $. Here $\nabla ^{\perp }=\left(
-\partial _{y},\partial _{x}\right) $.

We consider the linearized equations near a shear $\left( \boldsymbol{v}%
_{0},\rho _{0}\right) $. Let $\boldsymbol{v}=$ $\nabla ^{\perp }\psi $ and $%
\rho $ be infinitesimal perturbations of velocity and density. The
linearized equations are
\begin{equation}
\rho _{0}\left[ \left( \partial _{t}+U(y)\partial _{x}\right) \boldsymbol{v}%
+\left( v^{y}\partial _{y}\right) \boldsymbol{v}_{0}\right] +\nabla p=\rho
\boldsymbol{g},  \label{eqn-linearized-momentum-Euler}
\end{equation}%
\begin{equation}
\left( \partial _{t}+U(y)\partial _{x}\right) \rho +v^{y}\rho _{0}^{\prime
}\left( y\right) =0.  \label{eqn-linearized-density-Euler}
\end{equation}%
\begin{equation*}
\nabla \cdot \boldsymbol{v}=0.
\end{equation*}%
Taking the curl of (\ref{eqn-linearized-momentum-Euler}), we get

\begin{align}
-\frac{\rho _{0}^{\prime }(y)}{\rho _{0}}\left[ U^{\prime }(y)\partial
_{x}\psi +\left( \partial _{t}+U(y)\partial _{x}\right) (-\partial _{y}\psi )%
\right] &  \label{eqn-stream-linearized-Euler} \\
+\left( \partial _{t}+U(y)\partial _{x}\right) \Delta \psi -U^{\prime \prime
}(y)\partial _{x}\psi & =-\partial _{x}\left( \frac{\rho }{\rho _{0}}\right)
g.  \notag
\end{align}%
The equation (\ref{eqn-linearized-density-Euler}) can be written as
\begin{equation}
\left( \partial _{t}+U(y)\partial _{x}\right) \frac{\rho }{\rho _{0}}%
=-\partial _{x}\psi \frac{\rho _{0}^{\prime }(y)}{\rho _{0}}.
\label{eqn-linearized-relative-density}
\end{equation}%
Consider Couette flow with an exponential density profile, that is, $U(y)=Ry$%
, $\rho _{0}(y)=Ae^{-\beta y}$. Then (\ref{eqn-stream-linearized-Euler})-(%
\ref{eqn-linearized-relative-density}) become
\begin{equation}
\beta \left[ R\partial _{x}-\left( \partial _{t}+Ry\partial _{x}\right)
\partial _{y}\right] \psi +\left( \partial _{t}+Ry\partial _{x}\right)
\Delta \psi =-\partial _{x}\left( \frac{\rho }{\rho _{0}}\right) g,
\label{originalR01}
\end{equation}

\begin{equation}
\left( \partial _{t}+Ry\partial _{x}\right) \left( \frac{\rho }{\rho _{0}}%
\right) =\beta \partial _{x}\psi .  \label{originalR02}
\end{equation}%
If $R\neq 0$, denote $B^{2}=\frac{\beta g}{R^{2}}$ to be the Richardson
number, $T=\frac{R\rho }{\beta \rho _{0}(y)}$ be the relative density
perturbation, $\omega =-\Delta \psi $ be the vorticity perturbation and let $%
t^{\prime }=Rt$. Then we have
\begin{align*}
-\beta \left[ \partial _{x}-\left( \partial _{t^{\prime }}+y\partial
_{x}\right) \partial _{y}\right] \psi +(\partial _{t^{\prime }}+y\partial
_{x})\omega & =B^{2}\partial _{x}T, \\
(\partial _{t^{\prime }}+y\partial _{x})T& =\partial _{x}\psi .
\end{align*}%
For convenience we still use $t$ for $t^{\prime }$. Thus the resulting
linearized system is
\begin{equation}
-\beta \left[ \partial _{x}-\left( \partial _{t}+y\partial _{x}\right)
\partial _{y}\right] \psi +(\partial _{t}+y\partial _{x})\omega
=B^{2}\partial _{x}T,  \label{finalEquationOri1}
\end{equation}%
\begin{equation}
(\partial _{t}+y\partial _{x})T=\partial _{x}\psi ,
\label{finalEquationOri2}
\end{equation}%
\begin{equation}
\omega =-\Delta \psi .  \label{finalEquationOri3}
\end{equation}

The system (\ref{finalEquationOri1})-(\ref{finalEquationOri3}) is rather
complicated. Many authors, including H\o iland (\cite{hoiland53}), Case (%
\cite{Case60}), Kuo (\cite{kuo}), Hartman (\cite{hartman74}), Chimonas (\cite%
{chimonas79}), Brown and Stewartson (\cite{brown-stewartson80}), Farrell and
Ioannou (\cite{farrel93}), chose to consider the Boussinesq approximation,
where the variation of density is ignored except for the gravity force term $%
\rho \boldsymbol{g}$. To apply the Boussinesq approximation, the density
perturbation should be relatively small compared with the constant density.
Under this approximation, the Euler momentum equation becomes
\begin{equation*}
\bar{\rho}\left( \partial _{t}\boldsymbol{v+}\left( \boldsymbol{v}\cdot
\nabla \right) \boldsymbol{v}\right) +\nabla p=\rho \boldsymbol{g},
\end{equation*}%
where $\bar{\rho}$ is a constant and $\rho $ is the variation of density.
The linearized Boussinesq equations near a shear flow $\left( U\left(
y\right) ,0\right) $ with the density variation profile $\rho _{0}\left(
y\right) $ is \
\begin{equation}
\left( \partial _{t}+U(y)\partial _{x}\right) \Delta \psi -U^{\prime \prime
}(y)\partial _{x}\psi =-\partial _{x}\left( \frac{\rho }{\bar{\rho}}\right)
g,  \label{eqn-linearized-stream-Bou}
\end{equation}%
\begin{equation}
\left( \partial _{t}+U(y)\partial _{x}\right) \frac{\rho }{\bar{\rho}}%
=-\partial _{x}\psi \frac{\rho _{0}^{\prime }}{\bar{\rho}}.
\label{eqn-linearized-density}
\end{equation}%
Compared this with the linearized original equation (\ref%
{eqn-stream-linearized-Euler}), it can be regarded as the case when ${\rho
_{0}^{\prime }}/{\rho _{0}}$ is very small, such that the first term of (\ref%
{eqn-stream-linearized-Euler}) is neglected and ${\rho _{0}}$ is taken to be
a constant $\bar{\rho}$. For Couette flow $U(y)=Ry$ with the exponential
profile $\rho _{0}=Ae^{-\beta y}$, to use the Boussinesq approximation, $%
\beta $ should be small which implies that $\rho _{0}\thickapprox A\left(
1-\beta y\right) $. Thus, we consider the linearized Boussinesq equations
near Couette flow $\left( Ry,0\right) $ with the linear density variation
profile $\rho _{0}\left( y\right) =-A\beta y$ and a constant density
background $\bar{\rho}=A$. Then (\ref{eqn-linearized-stream-Bou})-(\ref%
{eqn-linearized-density}) become
\begin{equation}
\left( \partial _{t}+Ry\partial _{x}\right) \Delta \psi =-\partial
_{x}\left( \frac{\rho }{A}\right) g,  \label{boussinesqR01}
\end{equation}%
\begin{equation}
\left( \partial _{t}+Ry\partial _{x}\right) \left( \frac{\rho }{A}\right)
=\beta \partial _{x}\psi .  \label{boussinesqR02}
\end{equation}%
If $R\neq 0$, denoting $B^{2}=\frac{\beta g}{R^{2}},\ T=\frac{R\rho }{\beta A%
}$ and scaling the time $t$ by $Rt$, then we have
\begin{equation}
(\partial _{t}+y\partial _{x})\omega =B^{2}\partial _{x}T,
\label{finalEquation1}
\end{equation}%
\begin{equation}
(\partial _{t}+y\partial _{x})T=\partial _{x}\psi ,  \label{finalEquation2}
\end{equation}%
\begin{equation}
\omega =-\Delta \psi .  \label{finalEquation3}
\end{equation}

\subsection{Sobolev spaces}

Without loss of generality, from now on we assume period length $L$ in $x$ direction is $2\pi$. Define the Fourier transform of $f(x,y)\ \left( (x,y)\in \mathbb{T}\times
\mathbb{R}\right) $, as
\begin{equation*}
\hat{f}(k,\eta )=\frac{1}{2\pi }\int_{\mathbb{T}\times \mathbb{R}%
}e^{-ixk-iy\eta }f(x,y)\mathrm{d}x\mathrm{d}y,\ \ (k,\eta )\in \mathbb{Z}%
\times \mathbb{R}.
\end{equation*}%
Fourier inversion formula is
\begin{equation*}
f(x,y)=\frac{1}{2\pi }\sum_{k\in \mathbb{Z}}\int_{\mathbb{R}}e^{ixk+iy\eta }%
\hat{f}(k,\eta )\mathrm{d}x\mathrm{d}y.
\end{equation*}%
The Sobolev space $H_{x}^{s_{x}}H_{y}^{s_{y}}$ is defined to be all
functions $f$ in $L^{2}\left( \mathbb{T}\times \mathbb{R}\right) $
satisfying
\begin{equation*}
\sum_{k\in \mathbb{Z}}(1+k^{2})^{s_{x}}\int_{\mathbb{R}}\left( 1+\eta
^{2}\right) ^{s_{y}}\left\vert \hat{f}(k,\eta )\right\vert ^{2}\mathrm{d}%
\eta <+\infty ,
\end{equation*}%
with the norm
\begin{equation*}
\Vert f\Vert _{H_{x}^{s_{x}}H_{y}^{s_{y}}}=\left( \sum_{k\in \mathbb{Z}%
}(1+k^{2})^{s_{x}}\int_{\mathbb{R}}\left( 1+\eta ^{2}\right)
^{s_{y}}\left\vert \hat{f}(k,\eta )\right\vert ^{2}\mathrm{d}\eta \right) ^{%
\frac{1}{2}}.
\end{equation*}%
Similarly, we define
\begin{equation*}
\Vert f\Vert _{H_{x}^{s_{x}}W_{y}^{s_{y},p}}=\left( \sum_{k\in \mathbb{Z}%
}(1+k^{2})^{s_{x}}\Vert \hat{f}(k,y)\Vert _{W_{y}^{s_{y},p}}^{2}\right) ^{%
\frac{1}{2}},
\end{equation*}%
where $W_{y}^{s_{y},p}$ is the $L^{p}$ Sobolev space in $\mathbb{R}$\ and
\begin{equation*}
\hat{f}(k,y)=\frac{1}{\sqrt{2\pi} }\int_{\mathbb{T}}e^{-ixk}f(x,y)\mathrm{d}x,\ \
k\in \mathbb{Z}.
\end{equation*}

\subsection{Hypergeometric Functions}

Gaussian hypergeometric function $F(a,b;c;z)$ is defined by the power series
\begin{equation*}
F(a,b;c;z)=\sum_{n=0}^{\infty }\frac{(a)_{n}(b)_{n}}{(c)_{n}}\frac{z^{n}}{n!}
\end{equation*}%
for $|z|<1$, where
\begin{equation*}
(x)_{n}=\left\{
\begin{array}{lc}
1 & n=0, \\
x(x+1)\cdots (x+n-1) & n>0.%
\end{array}%
\right.
\end{equation*}%
Its value $F(z)$ for $|z|\geq 1$ is defined by the analytic continuation. If
$c,z\in \mathbb{R}$, and $a,b$ are complex conjugate, then $F(a,b;c;z)$ is
also real. The following lemma is known as Gauss' contiguous relation.

\begin{lem}
\label{gaussContiguousRelations} The derivative of $F(z)=F(a,b;c;z)$ can be
expressed as
\begin{align*}
\frac{dF}{dz}&=\frac{ab}{c}F(a+1,b+1;c+1;z) \\
&=\frac{c-1}{z}\left(F(a,b;c-1;z)-F(a,b;c;z)\right) \\
&=\frac{1}{c(1-z)}\left[\left(c-a\right)\left(c-b\right)F(a,b;c+1;z)+c%
\left(a+b-c\right)F(a,b;c;z)\right].
\end{align*}
\end{lem}

Hypergeometric functions are related to solutions of Euler's hypergeometric
differential equation.

\begin{lem}
\label{eulerEquation} Assume $c$ is not an integer. Euler's hypergeometric
differential equation
\begin{equation}
z(1-z)f^{\prime \prime }(z)+\left[ c-(a+b+1)z\right] f^{\prime }(z)-abf(z)=0
\label{eqn-Euler-hypergeometric}
\end{equation}%
has two linearly independent solutions
\begin{align*}
f_{1}(z)& =F(a,b;c;z), \\
f_{2}(z)& =z^{1-c}F(1+a-c,1+b-c;2-c;z).
\end{align*}
\end{lem}

The proof of these two lemmas can be found in pages 57 and 74 of the book (%
\cite{bateman53}).

Hypergeometric functions have one branch point at $z=1$, and another at $%
z=\infty $. The default cut-line connecting these two branch points is
chosen as $z>1,z\in \mathbb{R}$. Pfaff transform can relate the value of a
hypergeometric functions near $z=1$ to the value of another one near $%
z=\infty $ in the following way:
\begin{equation}
F(a,b;c;z)=(1-z)^{-b}F\left( c-a,b;c;\frac{z}{z-1}\right) ,  \label{Pfaff1}
\end{equation}%
\begin{equation}
F(a,b;c;z)=(1-z)^{-b}F\left( c-a,b;c;\frac{z}{z-1}\right) .  \label{Pfaff2}
\end{equation}%
By combining these two transforms, we obtain the Euler transform
\begin{equation}
F(a,b;c;z)=(1-z)^{c-a-b}F\left( c-a,c-b;c;z\right) .  \label{euler transform}
\end{equation}

When $\mathrm{Re}(c)>\mathrm{Re}(a+b)$ we have the Gauss formula
\begin{equation}
F(a,b;c;1)=\frac{\Gamma (c)\Gamma (c-a-b)}{\Gamma (c-a)\Gamma (c-b)}.
\label{gauss}
\end{equation}%
When $\mathrm{Re}(c)<\mathrm{Re}(a+b)$, $F(a,b;c;1)$ is infinity.

The following lemma plays an important role in solving the linearized
equations in the next Section.

\begin{lem}
\label{lemma-wronski} The Wronskian of the two solutions listed above is
\begin{equation*}
W(z)=f_{1}(z)f_{2}^{\prime }(z)-f_{1}^{\prime
}(z)f_{2}(z)=(1-c)z^{-c}(1-z)^{c-1-a-b}.
\end{equation*}
\end{lem}

\begin{proof}
By Liouville's formula, the Wronskian of Euler's hypergeometric differential
equation (\ref{eqn-Euler-hypergeometric}) can be written as
\begin{align*}
W(z)& =C\exp \left( -\int \frac{c-(a+b+1)z}{z(1-z)}\mathrm{d}z\right)  \\
& =C\exp \left( -\log (1-z)(a+b+1-c)-c\log (z)\right)  \\
& =Cz^{-c}(1-z)^{c-1-a-b}=Cz^{-c}+O(z^{-c-1})
\end{align*}%
To determine the constant $C$, it is sufficient to calculate the leading
order term of $W(z)$ in the power series expansion near $z=0$. By the
definition,
\begin{equation*}
f_{1}(0)=1,\ \ f_{1}^{\prime }(0)=\frac{ab}{c},\ \ f_{2}(z)\sim z^{1-c},\ \
f_{2}^{\prime}(z)\sim (1-c)z^{-c}
\end{equation*}%
when $z \rightarrow 0$, so $C=1-c$ and $W(z)=(1-c)z^{-c}(1-z)^{c-1-a-b}$.
\end{proof}

\section{Solutions by Hypergeometric functions}

In this section, we apply Fourier transform on the linearized systems (\ref%
{finalEquation1}-\ref{finalEquation3}) based on the Boussinesq approximation
and (\ref{finalEquationOri1}-\ref{finalEquationOri3}) based on full Euler
equations respectively. Then we reduce them to a second order ODE in $t$,
and solve it explicitly by using hypergeometric functions. We will study
these equations in the sheared coordinates $\left( z,y\right) =\left(
x-ty,y\right) $ and define
\begin{align*}
f(t;z,y)& =\omega (t;z+ty,y)=\omega (t;x,y), \\
\phi (t;z,y)& =\psi (t;z+ty,y)=\psi (t;x,y), \\
\tau (t;z,y)& =T(t;z+ty,y)=T(t;x,y).
\end{align*}

\subsection{Boussinesq approximation}

In the new coordinates $\left( z,y\right) $, equations (\ref{finalEquation1}-%
\ref{finalEquation3}) become the following:
\begin{align*}
\partial _{t}f(t;z,y)& =\left( \partial _{t}+y\partial _{x}\right) \omega
(t;x,y)=B^{2}\partial _{x}T(t;x,y)=B^{2}\partial _{z}\tau (t;z,y), \\
\partial _{t}\tau (t;z,y)& =\left( \partial _{t}+y\partial _{x}\right)
T(t;x,y)=\partial _{x}\psi (t;x,y)=\partial _{z}\phi (t;z,y),
\end{align*}%
\begin{equation*}
\left[ \partial _{zz}+(\partial _{y}-t\partial _{z})^{2}\right] \phi
(t;z,y)=\psi _{xx}+\psi _{yy}=-\omega (t;x,y)=-f(t;z,y).
\end{equation*}

By the Fourier transform $(z,y)\rightarrow (k,\eta )$, we get
\begin{equation*}
\hat{f}_{t}=B^{2}(ik)\hat{\tau},\ \ \hat{\tau}_{t}=(ik)\hat{\phi},
\end{equation*}%
\begin{equation}
\left[ (ik)^{2}+(i\eta -ikt)^{2}\right] \hat{\phi}=-\hat{f}.
\label{fourier-psi-f}
\end{equation}%
Differentiate (\ref{fourier-psi-f}) twice with respect to $t$ to get%
\begin{equation}
\left[ (ik)^{2}+(i\eta -ikt)^{2}\right] \hat{\phi}_{t}+2(i\eta -ikt)(-ik)%
\hat{\phi}=-\hat{f}_{t}=-B^{2}(ik)\hat{\tau},  \label{f_t}
\end{equation}%
\begin{align*}
& \left[ (ik)^{2}+(i\eta -ikt)^{2}\right] \hat{\phi}_{tt}+4(i\eta -ikt)(-ik)%
\hat{\phi}_{t}+2(-ik)^{2}\hat{\phi} \\
=& -\hat{f}_{tt}=-B^{2}(ik)\hat{\tau}_{t}=-B^{2}(ik)^{2}\hat{\phi}.
\end{align*}%
For fixed $k\neq 0$ and $\eta $, define $s=t-\frac{\eta }{k}$ and $s_{0}=-%
\frac{\eta }{k}$. Then we obtain a second order linear ODE for $\hat{\phi}$
\begin{equation}
(1+s^{2})\hat{\phi}_{tt}+4s\hat{\phi}_{t}+(2+B^{2})\hat{\phi}=0.
\label{ODE-phi}
\end{equation}%
First, we look for special solutions of the form $\hat{\phi}(t;k,\eta
)=g(-s^{2})$. Let $u=-s^{2}$, then $\hat{\phi}_{t}=-2sg^{\prime }$ and $\hat{%
\phi}_{tt}=4s^{2}g^{\prime \prime }-2g^{\prime }$. Equation (\ref{ODE-phi})
becomes
\begin{equation*}
u(1-u)g^{\prime \prime }+\left( \frac{1}{2}-\frac{5}{2}u\right) g^{\prime }-%
\frac{2+B^{2}}{4}g=0.
\end{equation*}%
This is in the form of Euler's hypergeometric differential equation (\ref%
{eqn-Euler-hypergeometric}) with $c=\frac{1}{2}$ and $a,b=\frac{3}{4}\pm
\frac{\nu }{2}$, where $\nu =\sqrt{\frac{1}{4}-B^{2}}$. By Lemma \ref%
{eulerEquation}, it has two linearly independent solutions%
\begin{equation}
g_{1}(s)=F\left( a,b;c;u\right) =F\left( \frac{3}{4}-\frac{\nu }{2},\frac{3}{%
4}+\frac{\nu }{2};\frac{1}{2};-s^{2}\right) ,  \label{g_1}
\end{equation}%
\begin{equation}
g_{2}(s)=-iu^{1-c}F\left( 1+a-c,1+b-c;2-c;u\right) =sF\left( \frac{5}{4}-%
\frac{\nu }{2},\frac{5}{4}+\frac{\nu }{2};\frac{3}{2};-s^{2}\right) .
\label{g_2}
\end{equation}%
Therefore, the general solutions to the equation (\ref{ODE-phi}) can be
written as
\begin{equation}
\hat{\phi}=C_{1}g_{1}(s)+C_{2}g_{2}(s),  \label{formula-solution-Boussinesq}
\end{equation}%
where $C_{1},C_{2}$ are some constants depending only on $(k,\eta )$. Note
that although a hypergeometric function has a branch point or singularity at
$z=1$, we only need its value at $z=-s^{2}$ which lies on the negative real
axis. Therefore, there is no ambiguity or singularity in (\ref%
{formula-solution-Boussinesq}).

The coefficients $C_{1},C_{2}$ are determined by the initial conditions $%
\psi (0;x,y)$ and $T(0;x,y)$. Let $\hat{\psi}^{0}(k,\eta ),\hat{T}%
^{0}(k,\eta )$ be the Fourier transforms of $\psi (0;x,y)$ and $T(0;x,y)$.
First,
\begin{equation*}
\hat{\phi}(0;k,\eta )=\hat{\phi}^{0}(k,\eta )=\hat{\psi}^{0}(k,\eta ),
\end{equation*}%
and by equation (\ref{f_t}),
\begin{equation*}
\hat{f}_{t}=k^{2}(1+s^{2})\hat{\phi}_{t}+2k^{2}s\hat{\phi}.
\end{equation*}%
Noticing that when $t=0$, $s=-\frac{\eta }{k}=s_{0}$,\ so we have
\begin{align*}
\hat{\phi}_{t}(0;k,\eta )& =\frac{\hat{f}_{t}(0;k,\eta )-2k^{2}s_{0}\hat{\phi%
}(0;k,\eta )}{k^{2}(1+s_{0}^{2})}=\frac{B^{2}(ik)\hat{\tau}(0;k,\eta
)-2k^{2}s_{0}\hat{\phi}(0;k,\eta )}{k^{2}(1+s_{0}^{2})} \\
& =\frac{1}{1+s_{0}^{2}}\left( \frac{iB^{2}}{k}\hat{\tau}^{0}-2s_{0}\hat{\phi%
}^{0}\right) =\frac{1}{1+s_{0}^{2}}\left( \frac{iB^{2}}{k}\hat{T}^{0}-2s_{0}%
\hat{\psi}^{0}\right) .
\end{align*}%
Now we have a linear system for $\left( C_{1},C_{2}\right) $
\begin{align*}
C_{1}g_{1}(s_{0})+C_{2}g_{2}(s_{0})& =\hat{\psi}^{0}, \\
C_{1}g_{1}^{\prime }(s_{0})+C_{2}g_{2}^{\prime }(s_{0})& =\frac{1}{%
1+s_{0}^{2}}\left( \frac{iB^{2}}{k}\hat{T}^{0}-2s_{0}\hat{\psi}^{0}\right) .
\end{align*}%
Therefore, the coefficients are
\begin{align}
& \begin{aligned} C_{1}(k,\eta )=&\frac{1}{\Delta }\left[ g_{2}^{\prime
}(s_{0})+\frac{2s_{0}}{1+s_{0}^{2}}g_{2}(s_{0})\right] \hat{\psi}^{0}(k,\eta
)\\&+\frac{1}{\Delta }\left[ -\frac{iB^{2}}{1+s_{0}^{2}}g_{2}(s_{0})\right]
\frac{\hat{T}^{0}(k,\eta )}{k}, \end{aligned}  \label{formula-C1} \\
& \begin{aligned} C_{2}(k,\eta )=&\frac{1}{\Delta }\left[ -g_{1}^{\prime
}(s_{0})-\frac{2s_{0}}{1+s_{0}^{2}}g_{1}(s_{0})\right] \hat{\psi}^{0}(k,\eta
)\\&+\frac{1}{\Delta }\left[ \frac{iB^{2}}{1+s_{0}^{2}}g_{1}(s_{0})\right]
\frac{\hat{T}^{0}(k,\eta )}{k}, \label{formula-C2} \end{aligned}
\end{align}%
where by Lemma \ref{lemma-wronski}%
\begin{align*}
\Delta & =g_{1}(s_{0})g_{2}^{\prime }(s_{0})-g_{1}^{\prime
}(s_{0})g_{2}(s_{0}) \\
&=-i(-2s_{0})\left( 1-\frac{1}{2}\right) \left( -s_{0}^{2}\right) ^{-\frac{1%
}{2}}\left( 1+s_{0}^{2}\right) ^{-2}=\frac{1}{\left( 1+s_{0}^{2}\right) ^{2}}%
,
\end{align*}%
which is strictly positive for all $s_{0}\in \mathbb{R}$.

Thus the solution of (\ref{ODE-phi}) is given explicitly by
\begin{equation*}
\hat{\phi}(t;k,\eta )=C_{1}(k,\eta )g_{1}(s)+C_{2}(k,\eta )g_{2}(s).
\end{equation*}%
As for $\hat{\tau}$, from equation (\ref{f_t}), for $B^{2}>0$ we have%
\begin{align}
\hat{\tau}(t;k,\eta )=& -\frac{ik}{B^{2}}\left( (1+s^{2})\hat{\phi}_{t}+2s%
\hat{\phi}\right) ,  \label{decayOfDensity} \\
=& -\frac{ik}{B^{2}}\left[ (1+s^{2})\left( C_{1}(k,\eta )g_{1}^{\prime
}(s)+C_{2}(k,\eta )g_{2}^{\prime }(s)\right) \right.  \notag \\
& \left. +2s\left( C_{1}(k,\eta )g_{1}(s)+C_{2}(k,\eta )g_{2}(s)\right)
\right] .  \notag
\end{align}

\subsection{Full Euler Equations}

Now we solve the linearized systems (\ref{finalEquationOri1})-(\ref%
{finalEquationOri3}) based on the full Euler equations. With $f,\phi ,\tau $
defined at the beginning of this section, equations (\ref{finalEquationOri1}%
)-(\ref{finalEquationOri3}) turn into
\begin{equation}
-\beta \left[ \partial _{z}-\partial _{t}\left( \partial _{y}-t\partial
_{z}\right) \right] \phi +\partial _{t}f=B^{2}\partial _{z}\tau ,
\label{eqn-f-original}
\end{equation}%
\begin{equation*}
\partial _{t}\tau =\partial _{z}\phi ,\ \ \ -\left[ \partial _{zz}+(\partial
_{y}-t\partial _{z})^{2}\right] \phi =f.
\end{equation*}%
By the Fourier transform $(z,y)\rightarrow (k,\eta )$, (\ref{eqn-f-original}%
) becomes
\begin{equation}
-\beta \left[ ik-\partial _{t}\left( i\eta -ikt\right) \right] \hat{\phi}+%
\hat{f}_{t}=B^{2}(ik)\hat{\tau}.  \label{eqn-Fourier-f-1st-ori}
\end{equation}%
Differentiate above with respect to $t$, we get
\begin{equation*}
-\beta \left[ ik\partial _{t}-\partial _{tt}\left( i\eta -ikt\right) \right]
\hat{\phi}+\hat{f}_{tt}=B^{2}(ik)\hat{\tau}_{t}.
\end{equation*}%
Substituting
\begin{equation}
\hat{\tau}_{t}=(ik)\hat{\phi},\ \ \ \hat{f}=-\left[ (ik)^{2}+(i\eta -ikt)^{2}%
\right] \hat{\phi},  \label{eqn-Fourier-tao-ori}
\end{equation}%
we have
\begin{equation*}
\partial _{tt}\left[ k^{2}+(\eta -kt)^{2}+\beta (i\eta -ikt)\right] \hat{\phi%
}-\beta (ik)\hat{\phi}_{t}+B^{2}k^{2}\hat{\phi}=0.
\end{equation*}%
Define $\chi =e^{-\frac{1}{2}\beta y}\phi $, then $\hat{\phi}(k,\eta )=\hat{%
\chi}(k,\eta +\frac{1}{2}i\beta )$ and the above equation implies%
\begin{align*}
\partial _{tt}\left[ k^{2}+\left( \eta -\frac{1}{2}i\beta -kt\right)
^{2}+\beta \left( i\left( \eta -\frac{1}{2}i\beta \right) -ikt\right) \right]
\hat{\chi}& \\
-\beta (ik)\hat{\chi}_{t}+B^{2}k^{2}\hat{\chi}& =0,
\end{align*}%
After simplification, we have
\begin{equation*}
\partial _{tt}\left[ \frac{1}{4}\beta ^{2}+k^{2}+\left( \eta -kt\right) ^{2}%
\right] \hat{\chi}-i\beta k\hat{\chi}_{t}+B^{2}k^{2}\hat{\chi}=0.
\end{equation*}%
For $k\neq 0$, again define $s=t-\frac{\eta }{k},s_{0}=-\frac{\eta }{k}$,
then
\begin{equation*}
\partial _{tt}\left[ \left( \frac{1}{4}\beta ^{2}+k^{2}+k^{2}s^{2}\right)
\hat{\chi}\right] -i\beta k\hat{\chi}_{t}+B^{2}k^{2}\hat{\chi}=0.
\end{equation*}%
Define $m=\sqrt{\frac{1}{4}\beta ^{2}+k^{2}},\kappa =\frac{k}{m},\beta _{1}=%
\frac{\beta }{2m}$, then we have
\begin{equation*}
\partial _{tt}\left[ \left( m^{2}+k^{2}s^{2}\right) \hat{\chi}\right]
-i\beta k\hat{\chi}_{t}+B^{2}k^{2}\hat{\chi}=0,
\end{equation*}%
\begin{equation*}
\partial _{tt}\left[ \left( 1+\kappa ^{2}s^{2}\right) \hat{\chi}\right]
-2i\beta _{1}\kappa \hat{\chi}_{t}+B^{2}\kappa ^{2}\hat{\chi}=0.
\end{equation*}%
Set $u=-i\kappa s$, then
\begin{align*}
-\partial _{uu}\left( 1-u^{2}\right) \hat{\chi}-2\beta _{1}\hat{\chi}%
_{u}+B^{2}\hat{\chi}& =0, \\
\left( 1-u^{2}\right) \hat{\chi}_{uu}+(2\beta _{1}-4u)\hat{\chi}%
_{u}-(2+B^{2})\hat{\chi}& =0.
\end{align*}%
Define $v=\frac{1-u}{2}$, then
\begin{equation}
v\left( 1-v\right) \hat{\chi}_{vv}+(-\beta _{1}+2-4v)\hat{\chi}_{v}-(2+B^{2})%
\hat{\chi}=0,  \label{oDE2}
\end{equation}%
which is of the form of Euler's hypergeometric differential equation (\ref%
{eqn-Euler-hypergeometric}) with $c=2-\beta _{1}$ and $a,b=\frac{3}{2}\pm
\nu $, where $\nu =\sqrt{\frac{1}{4}-B^{2}}$. By Lemma \ref{eulerEquation},
it has two linear independent solutions,
\begin{align*}
g_{3}(s)& =F\left( \frac{3}{2}-\nu ,\frac{3}{2}+\nu ;2-\beta _{1};v\right)
=F\left( \frac{3}{2}-\nu ,\frac{3}{2}+\nu ;2-\beta _{1};\frac{1+i\kappa s}{2}%
\right) , \\
g_{4}(s)& =\left( \frac{1+i\kappa s}{2}\right) ^{-1+\beta _{1}}F\left( \frac{%
1}{2}+\beta _{1}-\nu ,\frac{1}{2}+\beta _{1}+\nu ;\beta _{1};\frac{1+i\kappa
s}{2}\right)
\end{align*}%
Therefore, the general solution to equation (\ref{oDE2}) is
\begin{equation*}
\hat{\chi}=C_{3}g_{3}(s)+C_{4}g_{4}(s),
\end{equation*}%
where $C_{3},C_{4}$ are constants depending only on $\left( k,\eta \right) $%
. Note that we only need values of $g_{1}$, $g_{2}\,$at $\frac{1}{2}+\frac{%
\kappa s}{2}i$ $\left( s\in \mathbb{R}\right) $, that is, on the line $\mathrm{%
Re}(z)=\frac{1}{2}$. Therefore, the branch point at $z=1\ $will not cause
any ambiguity or singularity.

The initial conditions $\psi (0;x,y)$ and $T(0;x,y)$ are used to determine
the coefficients $C_{3},C_{4}$. Denote $\mu =e^{-\frac{1}{2}\beta y}\tau $, $%
\Psi ^{0}=e^{-\frac{1}{2}\beta y}\psi ^{0},\Upsilon ^{0}=e^{-\frac{1}{2}%
\beta y}T^{0}$, then
\begin{equation*}
\hat{\chi}(0;k,\eta )=\hat{\phi}^{0}\left( k,\eta -\frac{1}{2}i\beta \right)
=\widehat{e^{-\frac{1}{2}\beta y}\psi ^{0}}=\hat{\Psi ^{0}}.
\end{equation*}%
By equations (\ref{eqn-Fourier-f-1st-ori}) and (\ref{eqn-Fourier-tao-ori}),
we have
\begin{equation*}
\hat{\phi}_{t}=\frac{1}{1+s^{2}-\frac{i\beta }{k}s}\left[ \left( \frac{%
2i\beta }{k}-2s\right) \hat{\phi}+\frac{iB^{2}}{k}\hat{\tau}\right] .
\end{equation*}%
Hence
\begin{align*}
\hat{\chi}_{t}(t;k,\eta ) &= \hat{\phi}_{t}\left( t;k,\eta -\frac{1}{2}%
i\beta \right) \\
&= \frac{1}{1+\left( s+\frac{i\beta }{2k}\right) ^{2}-\frac{i\beta }{k}%
\left( s+\frac{i\beta }{2k}\right) }\left[ \left( \frac{2i\beta }{k}-2s-2%
\frac{i\beta }{2k}\right) \hat{\chi}+\frac{iB^{2}}{k}\hat{\mu}\right] \\
&= \frac{1}{1+|\tilde{s}|^{2}}\left( \frac{iB^{2}}{k}\hat{\chi}-2\tilde{s}%
\hat{\mu}\right) ,
\end{align*}%
and%
\begin{equation*}
\hat{\chi}_{t}(0;k,\eta )=\frac{1}{1+|\tilde{s}_{0}|^{2}}\left( \frac{iB^{2}%
}{k}\hat{\Upsilon ^{0}}-2\tilde{s}_{0}\hat{\Psi ^{0}}\right) ,
\end{equation*}%
where $\tilde{s}=s-\frac{i\beta }{2k},\tilde{s}_{0}=s_{0}-\frac{i\beta }{2k}$%
.

So we have a linear system for $\left( C_{3},C_{4}\right) :$
\begin{align*}
C_{3}g_{3}(s_{0})+C_{4}g_{4}(s_{0})& =\hat{\Psi}^{0}, \\
C_{3}g_{3}^{\prime }(s_{0})+C_{4}g_{4}^{\prime }(s_{0})& =\frac{1}{1+|\tilde{%
s}_{0}|^{2}}\left( \frac{iB^{2}}{k}\hat{\Upsilon ^{0}}-2\tilde{s}_{0}\hat{%
\Psi ^{0}}\right) ,
\end{align*}

which gives
\begin{align*}
C_{3}(k,\eta )=& \frac{1}{\Delta }\left[ g_{4}^{\prime }(s_{0})+\frac{2%
\tilde{s}_{0}}{1+|\tilde{s}_{0}|^{2}}g_{4}(s_{0})\right] \hat{\Psi}%
^{0}(k,\eta ) \\
& \ \ +\frac{1}{\Delta }\left[ -\frac{iB^{2}}{1+|\tilde{s}_{0}|^{2}}%
g_{4}(s_{0})\right] \frac{\hat{\Upsilon}^{0}(k,\eta )}{k},
\end{align*}%
\begin{align*}
C_{4}(k,\eta )=& \frac{1}{\Delta }\left[ -g_{3}^{\prime }(s_{0})-\frac{2%
\tilde{s}_{0}}{1+|\tilde{s}_{0}|^{2}}g_{3}(s_{0})\right] \hat{\Psi}%
^{0}(k,\eta ) \\
& \ \ +\frac{1}{\Delta }\left[ \frac{iB^{2}}{1+|\tilde{s}_{0}|^{2}}%
g_{3}(s_{0})\right] \frac{\hat{\Upsilon}^{0}(k,\eta )}{k},
\end{align*}%
where by Lemma \ref{lemma-wronski}

\begin{align*}
\Delta & =g_{3}(s_{0})g_{4}^{\prime }(s_{0})-g_{3}^{\prime
}(s_{0})g_{4}(s_{0}) \\
& =\frac{\kappa i}{2}\left( -1+\beta _{1}\right) \left( \frac{1}{2}+\frac{%
\kappa s_{0}}{2}i\right) ^{-2+\beta _{1}}\left( \frac{1}{2}-\frac{\kappa
s_{0}}{2}i\right) ^{-2-\beta _{1}},
\end{align*}%
which is never zero, because $|\kappa |,\ \beta _{1}\in (0,1)$ by
definition. Moreover,
\begin{equation*}
|\kappa |\geq \frac{1}{\sqrt{\frac{1}{4}\beta ^{2}+1}},\ \ \ 1-\beta
_{1}\geq 1-\frac{\beta /2}{\sqrt{\frac{1}{4}\beta ^{2}+1}}
\end{equation*}%
are both uniformly bounded away from zero for all integers $k\neq 0$. Hence
\begin{equation*}
|\Delta |^{-1}=\left\vert \frac{1}{2}+\frac{\kappa s_{0}}{2}i\right\vert
^{4}\left\vert \frac{\kappa }{2}\right\vert ^{-1}\left( 1-\beta _{1}\right)
^{-1}\lesssim \left\langle s_{0}\right\rangle ^{4}.
\end{equation*}

By equations (\ref{eqn-Fourier-f-1st-ori}) and (\ref{eqn-Fourier-tao-ori}),
for $B^{2}>0$ we have
\begin{equation*}
\hat{\tau}(t;k,\eta )=-\frac{ik}{B^{2}}\left[ -\frac{2i\beta }{k}\hat{\phi}-%
\frac{i\beta }{k}s\hat{\phi}_{t}+(1+s^{2})\hat{\phi}_{t}+2s\hat{\phi}\right]
,
\end{equation*}%
and%
\begin{align*}
\hat{\mu}(t;k,\eta ) =& \hat{\tau}\left( t;k,\eta -\frac{1}{2}i\beta \right)
\\
=& -\frac{ik}{B^{2}}\left[ -\frac{2i\beta }{k}\hat{\chi}-\frac{i\beta }{k}%
\left( s+\frac{i\beta }{2k}\right) \hat{\chi}_{t}+\left( 1+\left( s+\frac{%
i\beta }{2k}\right) ^{2}\right) \hat{\chi}_{t} \right. \\
&\left.+2\left( s+\frac{i\beta }{2k}\right) \hat{\chi}\right] \\
=& -\frac{ik}{B^{2}}\left[ \left( 1+s^{2}+\frac{\beta ^{2}}{4k^{2}}\right)
\hat{\chi}_{t}+2\left( s-\frac{i\beta }{2k}\right) \hat{\chi}\right] \\
=& -\frac{ik}{B^{2}}\left[ \left( 1+|\tilde{s}|^{2}\right) \hat{\chi}_{t}+2%
\tilde{s}\hat{\chi}\right] .
\end{align*}

\section{Decay estimates in the case of Boussinesq approximation}

In this section, we use the solution formula obtained in the last section to
obtain the inviscid decay estimates in Theorem \ref{thm-boussinesq}, for
solutions of the linearized equations under Boussinesq approximation.

\subsection{The case $B^2 > 0$ and $B^2 \neq \frac{1}{4}$}

By expanding $g_{1}(s)$, $g_{2}(s),$ $g_{1}^{\prime }(s_{0})$, $%
g_{2}^{\prime }(s_{0})$ at infinity, we obtain the following asymptotics%
\begin{align}
& \begin{aligned} \label{asymptotics-g1} g_{1}(s) =& \sqrt{\pi }\left[
\frac{\Gamma (\nu )}{\Gamma (-\frac{1}{4}+\frac{\nu }{2})\Gamma
(\frac{3}{4}+\frac{\nu }{2})}s^{-\frac{3}{2}+\nu } \right.\\&\left.
+\frac{\Gamma (-\nu )}{\Gamma (-\frac{1}{4}-\frac{\nu }{2})\Gamma
(\frac{3}{4}-\frac{\nu }{2})}s^{-\frac{3}{2}-\nu }\right] +O\left(
|s|^{-\frac{7}{2}+Re(\nu )}\right) , \end{aligned} \\
& \begin{aligned} \label{asymptotics-g2} g_{2}(s) =& \frac{\sqrt{\pi
}}{2}\left[ \frac{\Gamma (\nu )}{\Gamma (\frac{1}{4}+\frac{\nu }{2})\Gamma
(\frac{5}{4}+\frac{\nu }{2})}s^{-\frac{3}{2}+\nu }
\right.\\&\left.+\frac{\Gamma (-\nu )}{\Gamma (\frac{1}{4}-\frac{\nu
}{2})\Gamma (\frac{5}{4}-\frac{\nu }{2})}s^{-\frac{3}{2}-\nu }\right]
+O\left( |s|^{-\frac{5}{2}+Re(\nu )}\right) , \end{aligned}
\end{align}%
\begin{align}
& \begin{aligned} \label{asymptotics-g1'} g_{1}^{\prime }(s_{0}) =&
2\sqrt{\pi }\left[ \frac{\left( -\frac{3}{4}+\frac{\nu }{2}\right) \Gamma
(\nu )}{\Gamma (-\frac{1}{4}+\frac{\nu }{2})\Gamma (\frac{3}{4} +\frac{\nu
}{2})}s_{0}^{-\frac{5}{2}+\nu } \right.\\&\left. +\frac{\left(
-\frac{3}{4}-\frac{\nu }{2}\right) \Gamma (-\nu )}{\Gamma
(-\frac{1}{4}-\frac{\nu }{2})\Gamma (\frac{3}{4}-\frac{\nu
}{2})}s_{0}^{-\frac{5}{2}-\nu }\right] +O\left( |s_{0}|^{-\frac{7}{2}+Re(\nu
)}\right) , \end{aligned} \\
& \begin{aligned} \label{asymptotics-g2'}g_{2}^{\prime }(s_{0}) =& \sqrt{\pi
}\left[ \frac{\left( -\frac{3}{4}+\frac{\nu }{2}\right) \Gamma (\nu
)}{\Gamma (\frac{1}{4}+\frac{\nu }{2})\Gamma (\frac{5}{4}+\frac{\nu
}{2})}s_{0}^{-\frac{5}{2}+\nu } \right.\\&\left. +\frac{\left(
-\frac{3}{4}-\frac{\nu }{2}\right) \Gamma (-\nu )}{\Gamma
(\frac{1}{4}-\frac{\nu }{2})\Gamma (\frac{5}{4}-\frac{\nu
}{2})}s_{0}^{-\frac{5}{2}-\nu }\right] +O\left( |s_{0}|^{-\frac{7}{2}+Re(\nu
)}\right) . \end{aligned}
\end{align}

For $B^{2}<\frac{1}{4}$ or $>\frac{1}{4}$, $\nu $ is real or pure imaginary.
We treat these cases separately.

\subsubsection{The case $0<B^{2}<\frac{1}{4}$}

In this case $\nu $ is a real number between $0$ and $\frac{1}{2}$. By using
the above asymptotics of $g_{1}\left( s\right) ,g_{2}\left( s\right) $, we
obtain bounds for the coefficients of $C_{1},C_{2}$ (defined in (\ref%
{formula-C1}), (\ref{formula-C2})). Since
\begin{align*}
\frac{1}{\Delta }\left[ g_{2}^{\prime }(s_{0})+\frac{2s_{0}}{1+s_{0}^{2}}%
g_{2}(s_{0})\right] \lesssim & \left\langle s_{0}\right\rangle
^{4}\left\langle s_{0}\right\rangle ^{-\frac{5}{2}+\nu }=\left\langle
s_{0}\right\rangle ^{\frac{3}{2}+\nu }, \\
\frac{1}{\Delta }\left[ -\frac{iB^{2}}{1+s_{0}^{2}}g_{2}(s_{0})\right]
\lesssim & \left\langle s_{0}\right\rangle ^{4}\left\langle
s_{0}\right\rangle ^{-\frac{7}{2}+\nu }=\left\langle s_{0}\right\rangle ^{%
\frac{1}{2}+\nu }, \\
\frac{1}{\Delta }\left[ -g_{1}^{\prime }(s_{0})-\frac{2s_{0}}{1+s_{0}^{2}}%
g_{1}(s_{0})\right] \lesssim & \left\langle s_{0}\right\rangle
^{4}\left\langle s_{0}\right\rangle ^{-\frac{5}{2}+\nu }=\left\langle
s_{0}\right\rangle ^{\frac{3}{2}+\nu }, \\
\frac{1}{\Delta }\left[ \frac{iB^{2}}{1+s_{0}^{2}}g_{1}(s_{0})\right]
\lesssim & \left\langle s_{0}\right\rangle ^{4}\left\langle
s_{0}\right\rangle ^{-\frac{7}{2}+\nu }=\left\langle s_{0}\right\rangle ^{%
\frac{1}{2}+\nu },
\end{align*}%
and
\begin{equation*}
|g_{1}(s)|,|g_{2}(s)|\lesssim \left\langle s\right\rangle ^{-\frac{3}{2}+\nu
},
\end{equation*}%
so we have
\begin{align*}
\left\vert C_{1}(k,\eta )\right\vert \lesssim & \left\langle
s_{0}\right\rangle ^{\frac{3}{2}+\nu }\left( \left\vert \hat{\psi}%
^{0}(k,\eta )\right\vert +\frac{\left\vert \hat{T}^{0}(k,\eta )\right\vert }{%
\left\langle s_{0}\right\rangle |k|}\right) , \\
\left\vert C_{2}(k,\eta )\right\vert \lesssim & \left\langle
s_{0}\right\rangle ^{\frac{3}{2}+\nu }\left( \left\vert \hat{\psi}%
^{0}(k,\eta )\right\vert +\frac{\left\vert \hat{T}^{0}(k,\eta )\right\vert }{%
\left\langle s_{0}\right\rangle |k|}\right) .
\end{align*}%
Therefore
\begin{align}
\left\vert \hat{\phi}(t;k,\eta )\right\vert =& \left\vert C_{1}(k,\eta
)g_{1}(s)+C_{2}(k,\eta )g_{2}(s)\right\vert  \label{fourier-psi} \\
\lesssim & \left\langle s\right\rangle ^{-\frac{3}{2}+\nu }\left\langle
s_{0}\right\rangle ^{\frac{3}{2}+\nu }\left( \left\vert \hat{\psi}%
^{0}(k,\eta )\right\vert +\frac{\left\vert \hat{T}^{0}(k,\eta )\right\vert }{%
\left\langle s_{0}\right\rangle |k|}\right) .  \notag
\end{align}%
To get the decay estimates in the physical space $\left( x,y\right) \ $from
above, we note that the term $\left\langle s\right\rangle ^{-\frac{3}{2}+\nu
}$ does not decay when $t\approx \frac{\eta }{k}\ \left( \text{i.e. }%
s\approx 0\right) $ and as compensation the additional regularity of initial
data is needed to ensure the decay. This is made precise in the following
lemma.

\begin{lem}
\label{lemma-decay}Assume that there exists $a>0$ and $b,c\in \mathbb{R}$
such that
\begin{equation}
\left\vert \hat{g}(t;k,\eta )\right\vert \lesssim \left\langle
s\right\rangle ^{-a}\left\langle s_{0}\right\rangle ^{b}\left\vert
k\right\vert ^{c}\left\vert \hat{h}(k,\eta )\right\vert ,\text{ }0\neq k\in
\mathbb{Z}, \eta \in \mathbb{R},  \label{fourier-estimate}
\end{equation}%
then
\begin{equation*}
\left\Vert P_{\neq 0}g\left( t\right) \right\Vert _{L^{2}\left( \mathbb{%
T\times R}\right) }\lesssim \left\langle t\right\rangle ^{-a}\left\Vert
h\right\Vert _{H_{x}^{c}H_{y}^{b+a}}.
\end{equation*}
\end{lem}

\begin{proof}
We have
\begin{align*}
\int_\mathbb{R} \left\vert \hat{g}(t;k,\eta )\right\vert ^{2}\mathrm{d} \eta
=&\int_{\left\vert s\right\vert =\left\vert t-\frac{\eta }{k}\right\vert
\geq \frac{1}{2}|t|}\left\vert \hat{g}(t;k,\eta )\right\vert ^{2}\mathrm{d}
\eta +\int_{\left\vert t-\frac{\eta }{k}\right\vert \leq \frac{1}{2}%
|t|}\left\vert \hat{g}(t;k,\eta )\right\vert ^{2}\mathrm{d} \eta \\
=&I_{1}+I_{2}.
\end{align*}

By (\ref{fourier-estimate}), we have
\begin{equation*}
I_{1}\lesssim \left<t\right>^{-2a}\int_{\left\vert t-\frac{\eta }{k}%
\right\vert \geq \frac{1}{2}|t|} \left\langle s_{0}\right\rangle
^{2b}\left\vert k\right\vert ^{2c}\left\vert \hat{h}(k,\eta )\right\vert ^{2}%
\mathrm{d} \eta .
\end{equation*}%
Since $\left\vert t-\frac{\eta }{k}\right\vert \leq \frac{1}{2}|t|$ implies $%
\left\vert s_{0}\right\vert =\left\vert \frac{\eta }{k}\right\vert \geq
\frac{1}{2}|t|$, so
\begin{equation*}
I_{2}\lesssim \left<t\right>^{-2a}\int_{\left\vert t-\frac{\eta }{k}%
\right\vert \leq \frac{1}{2}|t|} \left\langle s_{0}\right\rangle
^{2b+2a}\left\vert k\right\vert ^{2c}\left\vert \hat{h}(k,\eta )\right\vert
^{2}\mathrm{d} \eta .
\end{equation*}%
Thus
\begin{equation*}
\int_\mathbb{R} \left\vert \hat{g}(t;k,\eta )\right\vert ^{2}\mathrm{d} \eta
\lesssim \left<t\right>^{-2a}\int_\mathbb{R} \left\langle s_{0}\right\rangle
^{2b+2a}\left\vert k\right\vert ^{2c}\left\vert \hat{h}(k,\eta )\right\vert
^{2}\mathrm{d} \eta ,
\end{equation*}%
and
\begin{align*}
\left\Vert P_{\neq 0}g\left( t\right) \right\Vert ^{2} _{L^{2}\left( \mathbb{%
T\times R}\right)}=&\sum_{k\neq 0}\int_\mathbb{R} \left\vert \hat{g}%
(t;k,\eta )\right\vert ^{2}\mathrm{d} \eta \\
\lesssim & \left<t\right>^{-2a}\sum_{k\neq 0}\left\vert k\right\vert
^{2c}\int_\mathbb{R} \left\langle \eta \right\rangle ^{2b+2a}\left\vert \hat{%
h}(k,\eta )\right\vert ^{2}\mathrm{d} \eta \\
\lesssim & \left<t\right>^{-2a}\left\Vert h\right\Vert
_{H_{x}^{c}H_{y}^{b+a}}^{2}.
\end{align*}
\end{proof}

\bigskip Since the velocity perturbation
\begin{align*}
v^{x}(t;x,y)=& -\partial _{y}\psi (t;x,y)=(-\partial _{y}+t\partial
_{z})\phi (t;z,y), \\
v^{y}(t;x,y)=& \partial _{x}\psi (t;x,y)=\partial _{z}\phi (t;z,y),
\end{align*}%
so by (\ref{fourier-psi}), we have
\begin{align*}
\left\vert \hat{v}^{x}\left( t;k,\eta \right) \right\vert & =\left\vert iks%
\hat{\phi}(t;k,\eta )\right\vert \\
& \leq \left\langle s\right\rangle ^{-\frac{1}{2}+\nu }\left\langle
s_{0}\right\rangle ^{\frac{3}{2}+\nu }\left( \left\vert k\right\vert
\left\vert \hat{\psi}^{0}(k,\eta )\right\vert +\frac{\left\vert \hat{T}%
^{0}(k,\eta )\right\vert }{\left\langle s_{0}\right\rangle }\right) , \\
\left\vert \hat{v}^{y}\left( t;k,\eta \right) \right\vert & =\left\vert ik%
\hat{\phi}(t;k,\eta )\right\vert \\
& \leq \left\langle s\right\rangle ^{-\frac{3}{2}+\nu }\left\langle
s_{0}\right\rangle ^{\frac{3}{2}+\nu }\left( \left\vert k\right\vert
\left\vert \hat{\psi}^{0}(k,\eta )\right\vert +\frac{\left\vert \hat{T}%
^{0}(k,\eta )\right\vert }{\left\langle s_{0}\right\rangle }\right) .
\end{align*}%
From equation (\ref{decayOfDensity}) we know
\begin{align*}
\left\vert \hat{\tau}(t;k,\eta )\right\vert \leq & \left\vert \frac{k}{B^{2}}%
\right\vert \left[ (1+s^{2})\left\vert C_{1}(k,\eta )g_{1}^{\prime
}(s)+C_{2}(k,\eta )g_{2}^{\prime }(s)\right\vert \right. \\
\ \ & \left. +2|s|\left\vert C_{1}(k,\eta )g_{1}(s)+C_{2}(k,\eta
)g_{2}(s)\right\vert \right] \\
\lesssim & \left\langle s\right\rangle ^{-\frac{1}{2}+\nu }\left\langle
s_{0}\right\rangle ^{\frac{3}{2}+\nu }\left( |k|\left\vert \hat{\psi}%
^{0}(k,\eta )\right\vert +\frac{\left\vert \hat{T}^{0}(k,\eta )\right\vert }{%
\left\langle s_{0}\right\rangle }\right) .
\end{align*}%
By Lemma \ref{lemma-decay},
\begin{align*}
\Vert P_{\neq 0}v^{x}\Vert _{L^{2}}\lesssim & \left\langle t\right\rangle ^{-%
\frac{1}{2}+\nu }\left( \Vert \psi ^{0}\Vert _{H_{x}^{1}H_{y}^{2}}+\Vert
T^{0}\Vert _{L_{x}^{2}H_{y}^{1}}\right) , \\
\Vert v^{y}\Vert _{L^{2}}\lesssim & \left\langle t\right\rangle ^{-\frac{3}{2%
}+\nu }\left( \Vert \psi ^{0}\Vert _{H_{x}^{1}H_{y}^{3}}+\Vert T^{0}\Vert
_{L_{x}^{2}H_{y}^{2}}\right) ,
\end{align*}%
and
\begin{equation*}
\Vert P_{\neq 0}T(t;\cdot ,\cdot )\Vert _{L^{2}}=\Vert P_{\neq 0}\tau
(t;\cdot ,\cdot )\Vert _{L^{2}}\lesssim \left\langle t\right\rangle ^{-\frac{%
1}{2}+\nu }\left( \Vert \psi ^{0}\Vert _{H_{x}^{1}H_{y}^{2}}+\Vert
T^{0}\Vert _{L_{x}^{2}H_{y}^{1}}\right) .
\end{equation*}

\subsubsection{The case $B^{2}>\frac{1}{4}$}

In this case, $\nu =\sqrt{\frac{1}{4}-B^{2}}$ is pure imaginary. Then from (%
\ref{asymptotics-g1}-\ref{asymptotics-g2'}), we have
\begin{align*}
|g_{1}(s)|\lesssim & \left\langle s\right\rangle ^{-\frac{3}{2}},\
|g_{2}(s)|\lesssim \left\langle s\right\rangle ^{-\frac{3}{2}}, \\
|g_{1}^{\prime }(s_{0})|\lesssim & \left\langle s_{0}\right\rangle ^{-\frac{5%
}{2}},\ |g_{2}^{\prime }(s_{0})|\lesssim \left\langle s_{0}\right\rangle ^{-%
\frac{5}{2}}.
\end{align*}%
By similar calculations,
\begin{align*}
\Vert P_{\neq 0}v^{x}\Vert _{L^{2}}\lesssim & \left\langle t\right\rangle ^{-%
\frac{1}{2}}\left( \Vert \psi ^{0}\Vert _{H_{x}^{1}H_{y}^{2}}+\Vert
T^{0}\Vert _{L_{x}^{2}H_{y}^{1}}\right) , \\
\Vert v^{y}\Vert _{L^{2}}\lesssim & \left\langle t\right\rangle ^{-\frac{3}{2%
}}\left( \Vert \psi ^{0}\Vert _{H_{x}^{1}H_{y}^{3}}+\Vert T^{0}\Vert
_{L_{x}^{2}H_{y}^{2}}\right) , \\
\Vert P_{\neq 0}T\Vert _{L^{2}}\lesssim & \left\langle t\right\rangle ^{-%
\frac{1}{2}}\left( \Vert \psi ^{0}\Vert _{H_{x}^{1}H_{y}^{2}}+\Vert
T^{0}\Vert _{L_{x}^{2}H_{y}^{1}}\right) .
\end{align*}%
Since $T$ is just $\rho /A$ times a positive constant, this completes the
proof of Theorem \ref{thm-boussinesq}(i)-(ii).

\subsection{The case $B^{2}=\frac{1}{4}$}

When $B^{2}=\frac{1}{4}$, $\nu =0$, the asymptotic approximations (\ref%
{asymptotics-g1}) and (\ref{asymptotics-g2}) no longer hold true, but the
following expansions at infinity emerge instead,
\begin{align*}
g_{1}(s)& =F\left( \frac{3}{4},\frac{3}{4};\frac{1}{2};-s^{2}\right) \\
& =\frac{2\sqrt{\pi }}{\Gamma \left( -\frac{1}{4}\right) \Gamma \left( \frac{%
3}{4}\right) }s^{-\frac{3}{2}}\log \left( s\right) -\frac{2\sqrt{\pi }\left(
\gamma +\digamma \left( \frac{3}{4}\right) +2\right) }{\Gamma \left( -\frac{1%
}{4}\right) \Gamma \left( \frac{3}{4}\right) }s^{-\frac{3}{2}}+O\left( |s|^{-%
\frac{7}{2}}\right) , \\
g_{2}(s)& =sF\left( \frac{5}{4},\frac{5}{4};\frac{3}{2};-s^{2}\right) \\
& =\frac{\sqrt{\pi }}{\Gamma \left( \frac{1}{4}\right) \Gamma \left( \frac{5%
}{4}\right) }s^{-\frac{3}{2}}\log \left( s\right) -\frac{\sqrt{\pi }\left(
\gamma +\digamma \left( \frac{1}{4}\right) +2\right) }{\Gamma \left( \frac{1%
}{4}\right) \Gamma \left( \frac{5}{4}\right) }s^{-\frac{3}{2}}+O\left( |s|^{-%
\frac{7}{2}}\right)
\end{align*}%
where $\gamma $ is the Euler constant, $\digamma (x)=\frac{\Gamma ^{\prime
}(x)}{\Gamma (x)}$ is the digamma function. It can be seen that with the
logarithm function, both solutions decay a little bit slower than before.

Similarly, their derivatives also have different asymptotic approximations
\begin{align*}
g_{1}^{\prime }(s_{0})=& -\frac{9}{4}s_{0}F\left( \frac{7}{4},\frac{7}{4};%
\frac{3}{2};-s_{0}^{2}\right) \\
=& -\frac{3\sqrt{\pi }}{\Gamma \left( -\frac{1}{4}\right) \Gamma \left(
\frac{3}{4}\right) }s_{0}^{-\frac{5}{2}}\log \left( s_{0}\right) \\
& \ \ \ \ \ +\frac{3\sqrt{\pi }\left( \gamma +\digamma \left( \frac{3}{4}%
\right) +\frac{8}{3}\right) }{\Gamma \left( -\frac{1}{4}\right) \Gamma
\left( \frac{3}{4}\right) }s_{0}^{-\frac{5}{2}}+O\left( |s_{0}|^{-\frac{7}{2}%
}\right) , \\
g_{2}^{\prime }(s_{0})=& F\left( \frac{5}{4},\frac{5}{4};\frac{3}{2}%
;-s_{0}^{2}\right) -\frac{25}{12}s_{0}^{2}F\left( \frac{9}{4},\frac{9}{4};%
\frac{5}{2};-s_{0}^{2}\right) \\
=& -\frac{3\sqrt{\pi }}{2\Gamma \left( \frac{1}{4}\right) \Gamma \left(
\frac{5}{4}\right) }s_{0}^{-\frac{5}{2}}\log \left( s_{0}\right) \\
& \ \ \ \ \ +\frac{3\sqrt{\pi }\left( \gamma +\digamma \left( \frac{1}{4}%
\right) +\frac{8}{3}\right) }{2\Gamma \left( \frac{1}{4}\right) \Gamma
\left( \frac{5}{4}\right) }s_{0}^{-\frac{5}{2}}+O\left( |s_{0}|^{-\frac{7}{2}%
}\right) .
\end{align*}%
Therefore, we obtain the following estimates
\begin{align*}
\left\vert g_{1}(s)\right\vert \lesssim & \left\langle s\right\rangle ^{-%
\frac{3}{2}}\left\langle \log \left\langle s\right\rangle \right\rangle ,\
\left\vert g_{2}(s)\right\vert \lesssim \left\langle s\right\rangle ^{-\frac{%
3}{2}}\left\langle \log \left\langle s\right\rangle \right\rangle , \\
\left\vert g_{1}^{\prime }(s_{0})\right\vert \lesssim & \left\langle
s_{0}\right\rangle ^{-\frac{5}{2}}\left\langle \log \left\langle
s_{0}\right\rangle \right\rangle ,\ \left\vert g_{2}^{\prime
}(s_{0})\right\vert \lesssim \left\langle s_{0}\right\rangle ^{-\frac{5}{2}%
}\left\langle \log \left\langle s_{0}\right\rangle \right\rangle ,
\end{align*}%
and as a result%
\begin{align*}
\left\vert C_{1}(k,\eta )\right\vert \lesssim & \left\langle
s_{0}\right\rangle ^{\frac{3}{2}}\left\langle \log \left\langle
s_{0}\right\rangle \right\rangle \left( \left\vert \hat{\psi}^{0}(k,\eta
)\right\vert +\frac{\left\vert \hat{T}^{0}(k,\eta )\right\vert }{%
\left\langle s_{0}\right\rangle |k|}\right) , \\
\left\vert C_{2}(k,\eta )\right\vert \lesssim & \left\langle
s_{0}\right\rangle ^{\frac{3}{2}}\left\langle \log \left\langle
s_{0}\right\rangle \right\rangle \left( \left\vert \hat{\psi}^{0}(k,\eta
)\right\vert +\frac{\left\vert \hat{T}^{0}(k,\eta )\right\vert }{%
\left\langle s_{0}\right\rangle |k|}\right) .
\end{align*}%
Therefore, we have
\begin{align*}
\left\vert \hat{\phi}(t;k,\eta )\right\vert =& \left\vert C_{1}(k,\eta
)g_{1}(s)+C_{2}(k,\eta )g_{2}(s)\right\vert \\
\lesssim & \left\langle s\right\rangle ^{-\frac{3}{2}}\left\langle
s_{0}\right\rangle ^{\frac{3}{2}}\left\langle \log \left\langle
s\right\rangle \right\rangle \left\langle \log \left\langle
s_{0}\right\rangle \right\rangle \left( \left\vert \hat{\psi}^{0}(k,\eta
)\right\vert +\frac{\left\vert \hat{T}^{0}(k,\eta )\right\vert }{%
\left\langle s_{0}\right\rangle |k|}\right) ,
\end{align*}%
from which the estimates of $\left\vert \hat{v}^{x}\left( t;k,\eta \right)
\right\vert ,\ \left\vert \hat{v}^{y}\left( t;k,\eta \right) \right\vert $
and $\left\vert \hat{\tau}(t;k,\eta )\right\vert $ follow. Then the decay
rates of $v^{x},v^{y},\ T$ can be obtained similarly as in the proof of
Lemma \ref{lemma-decay}, so we only sketch it. Notice that for any $a\geq
\frac{1}{2}$, the function $h\left( x\right) =\frac{\left\langle
x\right\rangle ^{a}}{\left\langle \log \left\langle x\right\rangle
\right\rangle }$ is increasing for all $x\geq 0$. When $\left\vert
s\right\vert \leq \frac{1}{2}|t|$ (implying $\left\vert s_{0}\right\vert
\geq \frac{1}{2}|t|$), we have%
\begin{align*}
\left\langle s\right\rangle ^{-a}\left\langle s_{0}\right\rangle ^{\frac{3}{2%
}}\left\langle \log \left\langle s\right\rangle \right\rangle \left\langle
\log \left\langle s_{0}\right\rangle \right\rangle \leq & \left\langle
s_{0}\right\rangle ^{\frac{3}{2}}\left\langle \log \left\langle
s_{0}\right\rangle \right\rangle \leq \frac{h\left( s_{0}\right) }{h\left(
\frac{1}{2}t\right) }\left\langle s_{0}\right\rangle ^{\frac{3}{2}%
}\left\langle \log \left\langle s_{0}\right\rangle \right\rangle \\
\lesssim & \left\langle t\right\rangle ^{-a}\left\langle \log \left\langle
t\right\rangle \right\rangle \left\langle s_{0}\right\rangle ^{\frac{3}{2}%
+a}.
\end{align*}%
On the other hand, when $\left\vert s\right\vert \geq \frac{1}{2}|t|$, we
have
\begin{equation*}
\left\langle s\right\rangle ^{-a}\left\langle s_{0}\right\rangle ^{\frac{3}{2%
}}\left\langle \log \left\langle s\right\rangle \right\rangle \left\langle
\log \left\langle s_{0}\right\rangle \right\rangle \lesssim \left\langle
t\right\rangle ^{-a}\left\langle \log \left\langle t\right\rangle
\right\rangle \left\langle s_{0}\right\rangle ^{\frac{3}{2}+a},
\end{equation*}%
since $\left\langle \log \left\langle s_{0}\right\rangle \right\rangle \leq
\left\langle s_{0}\right\rangle ^{a}$. Similar to the proof of Lemma \ref%
{lemma-decay}, we get
\begin{align*}
\Vert P_{\neq 0}v^{x}\Vert _{L^{2}}\lesssim & \left\langle t\right\rangle ^{-%
\frac{1}{2}}\left\langle \log \left\langle t\right\rangle \right\rangle
\left( \Vert \psi ^{0}\Vert _{H_{x}^{1}H_{y}^{2}}+\Vert T^{0}\Vert
_{L_{x}^{2}H_{y}^{1}}\right) , \\
\Vert v^{y}\Vert _{L^{2}}\lesssim & \left\langle t\right\rangle ^{-\frac{3}{2%
}}\left\langle \log \left\langle t\right\rangle \right\rangle \left( \Vert
\psi ^{0}\Vert _{H_{x}^{1}H_{y}^{3}}+\Vert T^{0}\Vert
_{L_{x}^{2}H_{y}^{2}}\right) .
\end{align*}%
and%
\begin{equation*}
\Vert P_{\neq 0}T\Vert _{L^{2}}\lesssim \left\langle t\right\rangle ^{-\frac{%
1}{2}}\left\langle \log \left\langle t\right\rangle \right\rangle \left(
\Vert \psi ^{0}\Vert _{H_{x}^{1}H_{y}^{2}}+\Vert T^{0}\Vert
_{L_{x}^{2}H_{y}^{1}}\right) .
\end{equation*}

\subsection{The case $B^{2}=0$}

When $B^{2}=0$, that is, $\beta =0$, then by (\ref{boussinesqR01})-(\ref%
{boussinesqR02}), we get
\begin{align*}
\left( \partial _{t}+Ry\partial _{x}\right) \Delta \psi =& -\partial
_{x}\left( \frac{\rho }{A}\right) g, \\
\left( \partial _{t}+Ry\partial _{x}\right) \left( \frac{\rho }{A}\right) =&
0.
\end{align*}%
For convenience, we let $R=1$. Again, we define
\begin{align*}
f(t;z,y)=& \omega (t;z+ty,y)=\omega (t;x,y), \\
\phi (t;z,y)=& \psi (t;z+ty,y)=\psi (t;x,y), \\
\tau (t;z,y)=& \frac{\rho }{A}(t;z+ty,y)=\frac{\rho }{A}(t;x,y).
\end{align*}%
Then
\begin{equation*}
\partial _{t}f(t;z,y)=g\partial _{z}\tau (t;z,y),\ \partial _{t}\tau
(t;z,y)=0.
\end{equation*}%
So
\begin{align*}
\hat{\tau}(t;k,\eta )=& \hat{\tau}(0;k,\eta ), \\
\hat{f}\left( t;k,\eta \right) =& \hat{f}\left( 0;k,\eta \right) +tikg\hat{%
\tau}(0;k,\eta )=\hat{\omega}^{0}(k,\eta )+tikg\hat{\rho ^{0}}\left( k,\eta
\right) ,
\end{align*}%
where $\omega (0;x,y)=\omega ^{0}(x,y),\ \frac{\rho }{A}(0;x,y)=\rho
^{0}\left( x,y\right) $. Thus by (\ref{fourier-psi-f}), we get
\begin{align*}
\left\vert \hat{\phi}(t;k,\eta )\right\vert =& \frac{1}{k^{2}\left(
1+s^{2}\right) }\left\vert \hat{f}\left( t;\eta ,k\right) \right\vert \\
\lesssim & \left\langle s\right\rangle ^{-2}\left\langle s_{0}\right\rangle
^{2}\left\vert \hat{\psi}^{0}(k,\eta )\right\vert +|t|\frac{1}{\left\vert
k\right\vert }\left\langle s\right\rangle ^{-2}\left\vert \hat{\rho}%
^{0}(k,\eta )\right\vert .
\end{align*}%
Therefore
\begin{equation*}
\left\vert \hat{v}^{x}\left( t;k,\eta \right) \right\vert \lesssim
\left\langle s\right\rangle ^{-1}\left\langle s_{0}\right\rangle
^{2}\left\vert k\right\vert \left\vert \hat{\psi}^{0}(k,\eta )\right\vert
+|t|\left\langle s\right\rangle ^{-1}\left\vert \hat{\rho}^{0}(k,\eta
)\right\vert ,
\end{equation*}%
\begin{equation*}
\left\vert \hat{v}^{y}\left( t;k,\eta \right) \right\vert \lesssim
\left\langle s\right\rangle ^{-2}\left\langle s_{0}\right\rangle
^{2}\left\vert k\right\vert \left\vert \hat{\psi}^{0}(k,\eta )\right\vert
+|t|\left\langle s\right\rangle ^{-2}\left\vert \hat{\rho}^{0}(k,\eta
)\right\vert .
\end{equation*}%
By Lemma \ref{lemma-decay}, we get
\begin{equation*}
\Vert P_{\neq 0}v^{x}\Vert _{L^{2}}\lesssim \Vert \rho ^{0}\Vert
_{L_{x}^{2}H_{y}^{1}}+\left\langle t\right\rangle ^{-1}\Vert \psi ^{0}\Vert
_{H_{x}^{1}H_{y}^{3}},
\end{equation*}%
\begin{equation*}
\Vert v^{y}\Vert _{L^{2}}\lesssim \left\langle t\right\rangle ^{-1}\Vert
\rho ^{0}\Vert _{L_{x}^{2}H_{y}^{2}}+\left\langle t\right\rangle ^{-2}\Vert
\psi ^{0}\Vert _{H_{x}^{1}H_{y}^{4}}.
\end{equation*}%
Also, $\left\Vert \frac{\rho }{A}\right\Vert _{L^{2}}\left( t\right)
=\left\Vert \rho ^{0}\right\Vert $. When $\rho ^{0}\neq 0$, there is no
decay for $\frac{\rho }{A}$ and $P_{\neq 0}v^{x}$. When $\rho ^{0}=0$, we
get
\begin{equation*}
\Vert P_{\neq 0}v^{x}\Vert _{L^{2}}\lesssim \left\langle t\right\rangle
^{-1}\Vert \psi ^{0}\Vert _{H_{x}^{1}H_{y}^{3}},\ \Vert v^{y}\Vert
_{L^{2}}\lesssim \left\langle t\right\rangle ^{-2}\Vert \psi ^{0}\Vert
_{H_{x}^{1}H_{y}^{4}},
\end{equation*}%
which exactly recovers the linear decay results in \cite{Lin-zeng}\ for the
homogeneous fluids.

\begin{remark}
For small $B>0$, the decay rates for $\Vert P_{\neq 0}v^{x}\Vert _{L^{2}}$
and $\Vert v^{y}\Vert _{L^{2}}$ are $t^{-\frac{1}{2}+\nu }$ and $t^{-\frac{3%
}{2}+\nu}$ respectively even when $\rho ^{0}=0$. Hence, if $B\rightarrow 0+$
(i.e. $\nu \rightarrow \frac{1}{2}-$), surprisingly the decay rates are
almost one order slower than the case of homogeneous fluids $(B=0$). This
apparent gap is due to the vanishing of the coefficient of $\left\langle
s\right\rangle ^{-\frac{3}{2}+\nu }$ terms in the asymptotics of
hypergeometric functions (\ref{asymptotics-g1})-(\ref{asymptotics-g2'}).
\end{remark}

\section{Decay estimates for the full Euler equation}

In this section, we prove the decay estimates in Theorem \ref{thm-full-euler}
for the linearized system of the full Euler equation. The proof is very
similar to the Boussinesq case, so we only sketch it.

\subsection{The case $0<B^{2}<\infty $}

For each $B^{2}>0$, we can find similar bounds for
\begin{equation*}
\hat{\chi}=C_{3}(k,\eta )g_{3}(s)+C_{4}(k,\eta )g_{4}(s)
\end{equation*}%
as in the Boussinesq case. For $B^{2}>0$ and $B^{2}\neq \frac{1}{4}$, the
asymptotics of $g_{3},g_{4}\ $at $s=\infty $ are
\begin{align*}
g_{3}(s)=& \left( \frac{1}{2}-\frac{i\kappa s}{2}\right) ^{-1-\beta _{1}}%
\left[ \frac{\Gamma (2-\beta _{1})\Gamma (-2\nu )}{\Gamma \left( \frac{3}{2}%
-\nu \right) \Gamma \left( \frac{1}{2}-\beta _{1}-\nu \right) }\left( -\frac{%
i\kappa s}{2}\right) ^{-\frac{1}{2}+\beta _{1}-\nu }\right. \\
& \left. +\frac{\Gamma (2-\beta _{1})\Gamma (2\nu )}{\Gamma \left( \frac{3}{2%
}+\nu \right) \Gamma \left( \frac{1}{2}-\beta _{1}+\nu \right) }\left( -%
\frac{i\kappa s}{2}\right) ^{-\frac{1}{2}+\beta _{1}+\nu }+O\left( |s|^{-%
\frac{3}{2}+Re(\nu )}\right) \right] \\
g_{4}(s)=& \left( \frac{1}{2}+\frac{i\kappa s}{2}\right) ^{-1+\beta _{1}}%
\left[ \frac{\Gamma (\beta _{1})\Gamma (-2\nu )}{\Gamma \left( -\frac{1}{2}%
-\nu \right) \Gamma \left( \frac{1}{2}+\beta _{1}-\nu \right) }\left( -\frac{%
i\kappa s}{2}\right) ^{-\frac{1}{2}-\beta _{1}-\nu }\right. \\
& +\left. \frac{\Gamma (\beta _{1})\Gamma (2\nu )}{\Gamma \left( -\frac{1}{2}%
+\nu \right) \Gamma \left( \frac{1}{2}+\beta _{1}+\nu \right) }\left( -\frac{%
i\kappa s}{2}\right) ^{-\frac{1}{2}-\beta _{1}+\nu }+O\left( |s|^{-\frac{3}{2%
}+Re(\nu )}\right) \right] \\
g_{3}^{\prime }(s)=& \left( \frac{i\kappa }{2}\right) \left( \frac{1}{2}-%
\frac{i\kappa s}{2}\right) ^{-\beta _{1}}\left[ \frac{\left( \frac{3}{2}+\nu
\right) \Gamma (2-\beta _{1})\Gamma (-2\nu )}{\Gamma \left( \frac{3}{2}-\nu
\right) \Gamma \left( \frac{1}{2}-\beta _{1}-\nu \right) }\left( -\frac{%
i\kappa s}{2}\right) ^{-\frac{5}{2}+\beta _{1}-\nu }\right. \\
& \left. +\frac{\left( \frac{3}{2}-\nu \right) \Gamma (2-\beta _{1})\Gamma
(2\nu )}{\Gamma \left( \frac{3}{2}+\nu \right) \Gamma \left( \frac{1}{2}%
-\beta _{1}+\nu \right) }\left( -\frac{i\kappa s}{2}\right) ^{-\frac{5}{2}%
+\beta _{1}+\nu }+O\left( |s|^{-\frac{7}{2}+Re(\nu )}\right) \right] \\
g_{4}^{\prime }(s)=& \left( \frac{i\kappa }{2}\right) \left( \frac{1}{2}+%
\frac{i\kappa s}{2}\right) ^{\beta _{1}}\left[ \frac{\left( -\frac{3}{2}-\nu
\right) \Gamma (2-\beta _{1})\Gamma (-2\nu )}{\Gamma \left( -\frac{1}{2}-\nu
\right) \Gamma \left( \frac{1}{2}+\beta _{1}-\nu \right) }\left( \frac{%
i\kappa s}{2}\right) ^{-\frac{5}{2}-\beta _{1}-\nu }\right. \\
& \left. +\frac{\left( -\frac{3}{2}+\nu \right) \Gamma (2-\beta _{1})\Gamma
(2\nu )}{\Gamma \left( -\frac{1}{2}+\nu \right) \Gamma \left( \frac{1}{2}%
+\beta _{1}+\nu \right) }\left( \frac{i\kappa s}{2}\right) ^{-\frac{5}{2}%
-\beta _{1}+\nu }+O\left( |s|^{-\frac{7}{2}+Re(\nu )}\right) \right] .
\end{align*}%
For $B^{2}=\frac{1}{4}$, the expansions at $s=\infty $ are
\begin{align*}
g_{3}(s)=& \left( \frac{1}{2}-\frac{i\kappa s}{2}\right) ^{-\beta _{1}} \\
& \times \left[ \frac{2\Gamma (2-\beta )}{\sqrt{\pi }\Gamma \left( \frac{1}{2%
}-\beta \right) }\left( -\frac{i\kappa s}{2}\right) ^{-\frac{3}{2}+\beta
_{1}}\log \left( -\frac{i\kappa s}{2}\right) +O\left( |s|^{-\frac{3}{2}%
+\beta _{1}}\right) \right] , \\
g_{4}(s)=& \left( \frac{1}{2}+\frac{i\kappa s}{2}\right) ^{\beta _{1}} \\
& \times \left[ \frac{\Gamma (\beta )}{2\sqrt{\pi }\Gamma \left( \frac{1}{2}%
+\beta \right) }\left( -\frac{i\kappa s}{2}\right) ^{-\frac{3}{2}-\beta
_{1}}\log \left( -\frac{i\kappa s}{2}\right) +O\left( |s|^{-\frac{3}{2}%
-\beta _{1}}\right) \right] ,
\end{align*}%
\begin{align*}
g_{3}^{\prime }(s)=& \left( \frac{i\kappa }{2}\right) \left( \frac{1}{2}-%
\frac{i\kappa s}{2}\right) ^{-\beta _{1}} \\
& \times \left[ \frac{3\Gamma (2-\beta )}{\sqrt{\pi }\Gamma \left( \frac{1}{2%
}-\beta \right) }\left( -\frac{i\kappa s}{2}\right) ^{-\frac{5}{2}+\beta
_{1}}\log \left( -\frac{i\kappa s}{2}\right) +O\left( |s|^{-\frac{5}{2}%
+\beta _{1}}\right) \right] ,
\end{align*}%
\begin{align*}
g_{4}^{\prime }(s)=& \left( \frac{i\kappa }{2}\right) \left( \frac{1}{2}+%
\frac{i\kappa s}{2}\right) ^{\beta _{1}} \\
& \times \left[ \frac{3\Gamma (\beta )}{4\sqrt{\pi }\Gamma \left( \frac{1}{2}%
+\beta \right) }\left( -\frac{i\kappa s}{2}\right) ^{-\frac{5}{2}-\beta
_{1}}\log \left( -\frac{i\kappa s}{2}\right) +O\left( |s|^{-\frac{5}{2}%
-\beta _{1}}\right) \right] .
\end{align*}%
Thus, we have the same bounds for $\hat{\chi}$, that is,

\begin{equation*}
\left\vert \hat{\chi}(t;k,\eta )\right\vert \lesssim \left\langle
s\right\rangle ^{-\frac{3}{2}+\nu }\left\langle s_{0}\right\rangle ^{\frac{3%
}{2}+\nu }\left( \left\vert \hat{\Psi}^{0}(k,\eta )\right\vert +\frac{%
\left\vert \hat{\Upsilon}^{0}(k,\eta )\right\vert }{\left\langle
s_{0}\right\rangle |k|}\right) ,
\end{equation*}%
when $0<B^{2}<\frac{1}{4};$%
\begin{equation*}
\left\vert \hat{\chi}(t;k,\eta )\right\vert \lesssim \left\langle
s\right\rangle ^{-\frac{3}{2}}\left\langle s_{0}\right\rangle ^{\frac{3}{2}%
}\left( \left\vert \hat{\Psi}^{0}(k,\eta )\right\vert +\frac{\left\vert \hat{%
\Upsilon}^{0}(k,\eta )\right\vert }{\left\langle s_{0}\right\rangle |k|}%
\right) ,
\end{equation*}%
when $B^{2}>\frac{1}{4}$, and
\begin{equation*}
\left\vert \hat{\chi}(t;k,\eta )\right\vert \lesssim \left\langle
s\right\rangle ^{-\frac{3}{2}}\left\langle s_{0}\right\rangle ^{\frac{3}{2}%
}\left\langle \log \left\langle s\right\rangle \right\rangle \left\langle
\log \left\langle s_{0}\right\rangle \right\rangle \left( \left\vert \hat{%
\Psi}^{0}(k,\eta )\right\vert +\frac{\left\vert \hat{\Upsilon}^{0}(k,\eta
)\right\vert }{\left\langle s_{0}\right\rangle |k|}\right) ,
\end{equation*}%
when $B^{2}=\frac{1}{4}$.

Since
\begin{equation*}
e^{-\frac{1}{2}\beta y}v^{y}(t;x,y)=e^{-\frac{1}{2}\beta y}\partial _{x}\psi
(t;x,y)=\partial _{x}e^{-\frac{1}{2}\beta y}\phi (t;x-ty,y)=\partial
_{z}\chi (t;z,y),
\end{equation*}%
\begin{align*}
e^{-\frac{1}{2}\beta y}v^{x}(t;x,y)=& e^{-\frac{1}{2}\beta y}\left(
-\partial _{y}\psi (t;x,y)\right) =e^{-\frac{1}{2}\beta y}(-\partial
_{y}+t\partial _{z})\phi (t;z,y) \\
=& (-\partial _{y}+t\partial _{z})\left( e^{-\frac{1}{2}\beta y}\phi
(t;z,y)\right) -\frac{1}{2}\beta e^{-\frac{1}{2}\beta y}\phi (t;z,y) \\
=& \left( -\partial _{y}+t\partial _{z}-\frac{1}{2}\beta \right) \chi
(t;x,y),
\end{align*}%
the decay estimates for $e^{-\frac{1}{2}\beta y}v^{x}$ and $e^{-\frac{1}{2}%
\beta y}v^{y}$ (in Theorem \ref{thm-full-euler} (i)-(iii)) can be proved as
in the Boussinesq case. The decay of the density variation can be obtained
similarly.

\subsection{The case $B^{2}=0$}

When $B^{2}=0$, i.e., $\beta =0$, the linearized equations are exactly the
same as the Boussinesq case. Thus all the estimates are the same.

\section{Dispersive decay in the absence of shear}

The shear plays a crucial role in the inviscid damping. Without a shear, the
decay mechanism is totally different. When $B^{2}<\infty $, the decay of $%
\left\Vert \boldsymbol{v}\right\Vert _{L^{2}}$ is due to the mixing of
vorticity caused by the shear motion. When $B^{2}=\infty $, $\left\Vert
\boldsymbol{v}\right\Vert _{L^{2}}$ does not decay but we have the decay of $%
\left\Vert \boldsymbol{v}\right\Vert _{L^{\infty }}$ due to dispersive
effects of the linear waves in a stably stratified fluid.

\subsection{Boussinesq Case}

When there is no shear, i.e. $R=0$, $B^{2}=\infty $, the equations (\ref%
{boussinesqR01}-\ref{boussinesqR02}) become

\begin{equation*}
\partial _{t}\Delta \psi =-\partial _{x}\left( \frac{\rho }{A}\right) g,\ \
\ \ \ \ \partial _{t}\left( \frac{\rho }{A}\right) =\beta \partial _{x}\psi .
\end{equation*}%
Denote $T=\frac{\rho }{\beta A},$ then above equations become
\begin{equation}
\Delta \psi _{t}=-\partial _{x}T\beta g,  \label{eqn-RT-B-1}
\end{equation}%
\begin{equation}
\partial _{t}T=\partial _{x}\psi .  \label{eqn-RT-B-2}
\end{equation}

\subsubsection{The $L^{2}$ Stability}

Multiplying (\ref{eqn-RT-B-1}) by $\psi $ and then integrating by parts with
(\ref{eqn-RT-B-2}), we get the following invariant%
\begin{equation*}
\frac{\mathrm{d}}{\mathrm{d}t}\left( \beta g\iint T^{2}\mathrm{d}x\mathrm{d}%
y+\iint \left\vert \nabla \psi \right\vert ^{2}\mathrm{d}x\mathrm{d}y\right)
=0.
\end{equation*}%
This shows that in the $L^{2}$ norm, the perturbations of velocity and
density are Liapunov stable but do not decay. However, below we show that
their $L^{\infty }$ norms decay due to the dispersive effects.

\subsubsection{The $L^{\infty }$ Decay}

First, we solve (\ref{eqn-RT-B-1})-(\ref{eqn-RT-B-2}) by Fourier transforms.
Denote $N^{2}=\beta g$ to be the squared Brunt-V\"{a}is\"{a}l\"{a}
frequency. By Fourier transform $(x,y)\rightarrow (k,\eta )$,
\begin{equation}
\left( (i\eta )^{2}+(ik)^{2}\right) \hat{\psi} _{t}=-(ik)N^{2}\hat{T},
\label{eqn-RT-B-Fourier-1}
\end{equation}%
\begin{equation}
\hat{T}_{t}=(ik)\hat{\psi }.  \label{eqn-RT-B-Fourier-2}
\end{equation}%
Combining (\ref{eqn-RT-B-Fourier-1})-(\ref{eqn-RT-B-Fourier-2}), we get
\begin{equation*}
\frac{d^{2}}{dt^{2}}\hat{\psi}=-\lambda ^{2}\hat{\psi},
\end{equation*}%
where ${\lambda ^{2}(k,\eta )=\frac{k^{2}N^{2}}{k^{2}+\eta ^{2}}}$. For $k
\neq 0$, its solutions are
\begin{equation*}
\hat{\psi}(t)=C_{1}e^{i\lambda t}+C_{2}e^{-i\lambda t}.
\end{equation*}%
By initial conditions,
\begin{equation*}
\hat{\psi}(0)=C_{1}+C_{2}=\hat{\psi}^{0},\ \ \hat{\psi}^{\prime
}(0)=i\lambda (C_{1}-C_{2})=\frac{i\lambda ^{2}}{k}\hat{T}^{0},
\end{equation*}%
thus%
\begin{equation*}
C_{1,2}=\frac{1}{2}\left( \hat{\psi}^{0}\pm \frac{\lambda }{k}\hat{T}%
^{0}\right) .
\end{equation*}%
By (\ref{eqn-RT-B-Fourier-1}),
\begin{equation*}
\hat{T}=-\frac{ik}{\lambda ^{2}}\hat{\psi}_{t}=\frac{k}{\lambda }\left(
C_{1}e^{i\lambda t}-C_{2}e^{-i\lambda t}\right) .
\end{equation*}%
To prove the $L^{\infty }$ decay of solutions, we need two lemmas. We refer to Souganidis and Strauss (\cite{souganidis-strauss}) for a general discussion on such dispersive decay.

\begin{lem}
\label{lemmavandercorput} (Van der Corput) Let $h(x)$ be either convex or
concave on $[a,b]$ with $-\infty \leq a<b\leq \infty $. Then
\begin{equation}
\left\vert \int_{b}^{a}e^{ih(\eta )}\mathrm{d} \eta \right\vert \leq 2\left(
\min_{[a,b]}|h^{\prime }|\right) ^{-1},\ \ \left\vert \int_{b}^{a}e^{ih(\eta
)}\mathrm{d} \eta \right\vert \leq 4\left( \min_{[a,b]}|h^{\prime \prime
}|\right) ^{-\frac{1}{2}}.  \label{vanderCorput}
\end{equation}
\end{lem}

\begin{lem}
\label{oscillatingIntegral} For $\lambda (k,\eta )=\frac{|k|N}{\sqrt{%
k^{2}+\eta ^{2}}}$ and $n$ sufficiently large,
\begin{equation*}
\left\vert \int_{-n}^{n}e^{i(\lambda t+\eta y)}\mathrm{d} \eta \right\vert
\lesssim |k|^{\frac{3}{2}}|Nt|^{-\frac{1}{3}}+|Nt|^{-\frac{1}{2}}|k|^{-\frac{%
1}{2}}n^{\frac{3}{2}}.
\end{equation*}
\end{lem}

\begin{proof}
We can assume $N=1$ without loss of generality. Notice that
\begin{align*}
\lambda (\eta ) =&\frac{1}{\sqrt{1+\left( \frac{\eta }{k}\right) ^{2}}}%
=\left\langle \frac{\eta }{k}\right\rangle ^{-1}, \\
\lambda ^{\prime }(\eta ) =&-\frac{\eta }{k^{2}}\left\langle \frac{\eta }{k}%
\right\rangle ^{-3}, \\
\lambda ^{\prime \prime }(\eta ) =&\frac{2\eta ^{2}-k^{2}}{k^{4}}%
\left\langle \frac{\eta }{k}\right\rangle ^{-5},
\end{align*}%
and $\lambda (\eta )$ has two inflection point, $\eta _{1,2}=\pm \frac{\sqrt{%
2}}{2}k$. Let $n>\frac{\sqrt{2}}{2}|k|$. Choose $\epsilon >0$ so small that
all the Taylor's expansion below are valid in $\left( \eta _{i}-\epsilon
,\eta _{i}+\epsilon \right) ,\ i=1,2$. Define
\begin{equation*}
S_{1}=\left( -n,\eta _{1}-\epsilon \right) \cup \left( \eta _{1}+\epsilon
,\eta _{2}-\epsilon \right) \cup \left( \eta _{2}+\epsilon ,n\right) .
\end{equation*}%
By (\ref{vanderCorput}), we have
\begin{align*}
\left\vert \int_{S_{1}}e^{i(\lambda t+\eta y)}\mathrm{d} \eta \right\vert
\leq &4\left( \min_{[a,b]}|t||\lambda ^{\prime \prime }|\right) ^{-\frac{1}{%
2}} \\
=&4|t|^{-\frac{1}{2}}\left( \frac{2n^{2}-k^{2}}{k^{4}}\left\langle \frac{n}{%
k}\right\rangle ^{-5}\right) ^{-\frac{1}{2}} \\
\lesssim &|k|^{-\frac{1}{2}}|t|^{-\frac{1}{2}}n^{\frac{3}{2}},
\end{align*}%
provided $n=n(\epsilon )$ is sufficiently large. For large $t$, we can
divide $(\eta _{1}-\epsilon ,\eta _{1}+\epsilon )=\left\{ |t|^{-\frac{1}{3}%
}<|\eta -\eta _{1}|<\epsilon \right\} \cup \left\{ |\eta -\eta _{1}|\leq
|t|^{-\frac{1}{3}}\right\} =S_{2}\cup S_{3}$, so that
\begin{equation*}
\left\vert \int_{\eta _{1}-\epsilon }^{\eta _{1}+\epsilon }e^{i(\lambda
t+\eta y)}\mathrm{d} \eta \right\vert \leq 4|t|^{-\frac{1}{2}}\left(
\min_{S_{2}}|\lambda ^{\prime \prime }|\right) ^{-\frac{1}{2}}+2|t|^{-\frac{1%
}{3}}.
\end{equation*}%
For $\eta \in S_{2}$, we have
\begin{align*}
\left\vert \lambda ^{\prime \prime }(\eta )\right\vert =&\frac{\left\vert
2\eta ^{2}-k^{2}\right\vert }{k^{4}}\left\langle \frac{\eta }{k}%
\right\rangle ^{-5} \\
=&\frac{2\left\vert \eta -\eta _{1}\right\vert \left\vert \eta -\eta
_{2}\right\vert }{k^{4}}\left\langle \frac{\eta }{k}\right\rangle ^{-5} \\
>&\frac{2\left\vert \eta -\eta _{2}\right\vert }{k^{4}}\left\langle \frac{%
\eta }{k}\right\rangle ^{-5}|t|^{-\frac{1}{3}} \\
\gtrsim &|k|^{-3}|t|^{-\frac{1}{3}}.
\end{align*}%
Therefore
\begin{equation*}
\left\vert \int_{\eta _{1}-\epsilon }^{\eta _{1}+\epsilon }e^{i(\lambda
t+\eta y)}\mathrm{d} \eta \right\vert \lesssim 4|t|^{-\frac{1}{2}}\left(
|k|^{-3}|t|^{-\frac{1}{3}}\right) ^{-\frac{1}{2}}+2|t|^{-\frac{1}{3}%
}\lesssim |k|^{\frac{3}{2}}|t|^{-\frac{1}{3}}.
\end{equation*}%
Applying similar estimates to $(\eta _{2}-\epsilon ,\eta _{2}+\epsilon )$
will complete the proof of this lemma.
\end{proof}

Now we prove the $L^{\infty }$ decay of the solutions of (\ref{eqn-RT-B-1})-(%
\ref{eqn-RT-B-2}). By Fourier inverse transform formula,
\begin{align*}
P_{\neq 0}\psi (t;x,y)=& \frac{1}{2\pi }\sum_{k\neq 0}\left(
e^{ikx}\int_{-\infty }^{\infty }\hat{\psi}(t)e^{i\eta y}\mathrm{d}\eta
\right) \\
=& \frac{1}{2\pi }\sum_{k\neq 0}\left( e^{ikx}\int_{-\infty }^{\infty
}\left( C_{1}(k,\eta )e^{i\lambda t}+C_{2}(k,\eta )e^{-i\lambda t}\right)
e^{i\eta y}\mathrm{d}\eta \right) ,
\end{align*}%
where
\begin{align*}
& \left\vert \int_{-\infty }^{\infty }C_{1}(k,\eta )e^{i\lambda t}e^{i\eta y}%
\mathrm{d}\eta \right\vert \\
\leq & \frac{1}{2}\left\vert \int_{-\infty }^{\infty }\hat{\psi ^{0}}(k,\eta
)e^{i\lambda t}e^{i\eta y}\mathrm{d}\eta \right\vert +\frac{1}{2|k|}%
\left\vert \int_{-\infty }^{\infty }\lambda \hat{T^{0}}(k,\eta )e^{i\lambda
t}e^{i\eta y}\mathrm{d}\eta \right\vert .
\end{align*}%
Define
\begin{align*}
I(y)=& \int_{-n}^{n}e^{i\lambda (k,\eta )t}\hat{\psi ^{0}}(k,\eta )e^{i\eta
y}\mathrm{d}\eta \\
=& \sqrt{2\pi} \left( e^{i\lambda (k,\eta )t}\chi _{\lbrack -n,n]}\hat{\psi ^{0}}(k,\eta
)\right) ^{\vee }(y) \\
=& \left( e^{i\lambda (k,\eta )t}\chi _{\lbrack -n,n]}\right) ^{\vee }\ast
\hat{\psi ^{0}}(k,y),
\end{align*}%
then
\begin{align*}
\Vert I(y)\Vert _{L^{\infty }}\leq & \left\Vert \left( e^{i\lambda (k,\eta
)t}\chi _{\lbrack -n,n]}\right) ^{\vee }\right\Vert _{L_{y}^{\infty }}\Vert
\hat{\psi ^{0}}(k,\cdot )\Vert _{L_{y}^{1}} \\
\leq & \left\Vert \int_{-n}^{n}e^{i\lambda (k,\eta )t}e^{i\eta y}d\eta
\right\Vert _{L_{y}^{\infty }}\Vert \hat{\psi ^{0}}(k,\cdot )\Vert
_{L_{y}^{1}}.
\end{align*}%
Here, $^{\vee }$ stands for the inverse Fourier transform. By lemma \ref%
{oscillatingIntegral}, we have
\begin{align*}
& \left\vert \int_{-\infty }^{\infty }\hat{\psi ^{0}}(k,\eta )e^{i\lambda
t}e^{i\eta y}\mathrm{d}\eta \right\vert \\
\leq & \int_{|\eta |>n}\left\vert \hat{\psi ^{0}}(k,\eta )\right\vert
\mathrm{d}\eta +\left\vert I(y)\right\vert \\
\lesssim & \left( \int_{|\eta |>n}\left\langle \eta \right\rangle ^{-2\alpha
}\mathrm{d}\eta \right) ^{\frac{1}{2}}\Vert \hat{\psi ^{0}}(k,\cdot )\Vert
_{H_{y}^{\alpha }} \\
& +\left( |k|^{\frac{3}{2}}|Nt|^{-\frac{1}{3}}+|k|^{-\frac{1}{2}}|Nt|^{-%
\frac{1}{2}}n^{\frac{3}{2}}\right) \Vert \hat{\psi ^{0}}(k,\cdot )\Vert
_{L_{y}^{1}} \\
\lesssim & \left( n^{-\alpha +\frac{1}{2}}+|k|^{\frac{3}{2}}|t|^{-\frac{1}{3}%
}+|k|^{-\frac{1}{2}}|t|^{-\frac{1}{2}}n^{\frac{3}{2}}\right) \left( \Vert
\hat{\psi ^{0}}(k,\cdot )\Vert _{H_{y}^{\alpha }}+\Vert \hat{\psi ^{0}}%
(k,\cdot )\Vert _{L_{y}^{1}}\right) .
\end{align*}

Choose $n=|t|^{\frac{1}{2\alpha +2}}$, for $\alpha \in \left( \frac{1}{2},%
\frac{7}{2}\right] $, we have
\begin{equation*}
\left\vert \int_{-\infty }^{\infty }\hat{\psi ^{0}}(k,\eta )e^{i\lambda
t}e^{i\eta y}\mathrm{d}\eta \right\vert \lesssim |k|^{\frac{3}{2}}|t|^{-%
\frac{2\alpha -1}{4\alpha +4}}\left( \Vert \hat{\psi ^{0}}\Vert
_{H_{y}^{\alpha }}+\Vert \hat{\psi ^{0}}\Vert _{L_{y}^{1}}\right) .
\end{equation*}%
If the initial condition is smooth enough, then
\begin{equation*}
\left\vert \int_{-\infty }^{\infty }\hat{\psi ^{0}}(k,\eta )e^{i\lambda
t}e^{i\eta y}\mathrm{d}\eta \right\vert \lesssim |k|^{\frac{3}{2}}|t|^{-%
\frac{1}{3}}\left( \Vert \hat{\psi ^{0}}\Vert _{H_{y}^{7/2}}+\Vert \hat{\psi
^{0}}\Vert _{L_{y}^{1}}\right) .
\end{equation*}%
%
%
%
%
%
%
%
%
%
%
%
%
%
%
%
%
%
%
%
%
%
Similarly,
\begin{equation*}
\left\vert \int_{-\infty }^{\infty }\lambda \hat{T^{0}}(k,\eta )e^{i\lambda
t}e^{i\eta y}\mathrm{d}\eta \right\vert \lesssim N|k|^{\frac{3}{2}}|t|^{-%
\frac{1}{3}}\left( \Vert \hat{T^{0}}\Vert _{H_{y}^{7/2}}+\Vert \hat{T^{0}}%
\Vert _{L_{y}^{1}}\right) .
\end{equation*}%
Therefore, we have
\begin{align*}
\Vert P_{\neq 0}\hat{\psi}(t;k,\cdot )\Vert _{L_{y}^{\infty }}\lesssim &
|t|^{-\frac{1}{3}}\left( |k|^{\frac{3}{2}}\Vert \hat{\psi ^{0}}\Vert
_{H_{y}^{7/2}}+|k|^{\frac{3}{2}}\Vert \hat{\psi ^{0}}\Vert
_{L_{y}^{1}}\right. \\
& \left. +|k|^{\frac{1}{2}}\Vert \hat{T^{0}}\Vert _{H_{y}^{7/2}}+|k|^{\frac{1%
}{2}}\Vert \hat{T^{0}}\Vert _{L_{y}^{1}}\right) .
\end{align*}%
Hence the decay in $L_{x}^{2}L_{y}^{\infty }$ is obtained:
\begin{align*}
\Vert P_{\neq 0}\psi \Vert _{L_{x}^{2}L_{y}^{\infty }}\lesssim & |t|^{-\frac{%
1}{3}}\left( \Vert \psi ^{0}\Vert _{H_{x}^{3/2}H_{y}^{7/2}}+\Vert \psi
^{0}\Vert _{H_{x}^{3/2}L_{y}^{1}}\right. \\
& \left. +\Vert T^{0}\Vert _{H_{x}^{1/2}H_{y}^{7/2}}+\Vert T^{0}\Vert
_{H_{x}^{1/2}L_{y}^{1}}\right) , \\
\Vert P_{\neq 0}v^{x}\Vert _{L_{x}^{2}L_{y}^{\infty }}\lesssim & |t|^{-\frac{%
1}{3}}\left( \Vert \psi ^{0}\Vert _{H_{x}^{3/2}H_{y}^{9/2}}+\Vert \psi
^{0}\Vert _{H_{x}^{3/2}W_{y}^{1,1}}\right. \\
& \left. +\Vert T^{0}\Vert _{H_{x}^{1/2}H_{y}^{9/2}}+\Vert T^{0}\Vert
_{H_{x}^{1/2}W_{y}^{1,1}}\right) ,
\end{align*}%
\begin{align*}
\Vert v^{y}\Vert _{L_{x}^{2}L_{y}^{\infty }}\lesssim & |t|^{-\frac{1}{3}%
}\left( \Vert \psi ^{0}\Vert _{H_{x}^{5/2}H_{y}^{7/2}}+\Vert \psi ^{0}\Vert
_{H_{x}^{5/2}L_{y}^{1}}\right. \\
& \left. +\Vert T^{0}\Vert _{H_{x}^{3/2}H_{y}^{7/2}}+\Vert T^{0}\Vert
_{H_{x}^{3/2}L_{y}^{1}}\right) .
\end{align*}%
Similarly, for the density we have
\begin{align*}
\Vert P_{\neq 0}T\Vert _{L_{x}^{2}L_{y}^{\infty }}\lesssim & |t|^{-\frac{1}{3%
}}\left( \Vert \psi ^{0}\Vert _{H_{x}^{5/2}H_{y}^{9/2}}+\Vert \psi ^{0}\Vert
_{H_{x}^{5/2}W_{y}^{1,1}}\right. \\
& \left. +\Vert T^{0}\Vert _{H_{x}^{3/2}H_{y}^{7/2}}+\Vert T^{0}\Vert
_{H_{x}^{3/2}L_{y}^{1}}\right) .
\end{align*}

Below, we show that the decay rate $|t|^{-\frac{1}{3}}$ obtained above is
sharp by constructing an example. Recall that the solution to (\ref%
{eqn-RT-B-Fourier-1})-(\ref{eqn-RT-B-Fourier-2}) is
\begin{equation*}
\hat{\psi}(t;k,\eta )=C_{1}e^{i\lambda t}+C_{2}e^{-i\lambda t}.
\end{equation*}%
where $k\neq 0$, ${\lambda ^{2}(k,\eta )=\frac{k^{2}N^{2}}{k^{2}+\eta ^{2}}}$
and $C_{1,2}{(k,\eta )}$ are determined by $\hat{\psi}^{0},\hat{T}^{0}$.
Therefore, for a fixed $k$, we consider a function of the form
\begin{equation*}
\hat{\psi}(t;k,\eta )=f(\eta )e^{i\lambda t},
\end{equation*}%
where $f(\eta )$ is to be chosen below. By the Fourier inverse formula
\begin{equation*}
\psi (t;k,y)=\frac{1}{\sqrt{2\pi} }\int_{-\infty }^{\infty }f(\eta )e^{i\lambda
t+i\eta y}\mathrm{d}\eta .
\end{equation*}%
We will look at the value of $\psi \ $at $y=ct$ where $c$ is a constant to
be determined later. Define
\begin{equation*}
g(\eta ):=\lambda \left( \eta \right) +c\eta =\frac{kN}{\sqrt{k^{2}+\eta ^{2}%
}}+c\eta .
\end{equation*}%
We note that $\eta ^{\ast }=\frac{k}{\sqrt{2}}$ is one inflection point of $%
\lambda \left( \eta \right) \ $(the other one is $-\frac{k}{\sqrt{2}}$). Let
$c=\frac{2N}{3\sqrt{3}k},$ then $g^{\prime \prime }\left( \eta ^{\ast
}\right) =\lambda ^{\prime \prime }\left( \eta ^{\ast }\right) =0,$
\begin{equation*}
g^{\prime }\left( \eta ^{\ast }\right) =-\frac{\eta ^{\ast }N}{k^{2}}%
\left\langle \frac{\eta ^{\ast }}{k}\right\rangle ^{-3}+c=-\frac{2N}{3\sqrt{3%
}k}+c=0,
\end{equation*}%
and%
\begin{eqnarray*}
g^{\prime \prime \prime }\left( \eta ^{\ast }\right) &=&-\frac{N}{k^{3}}%
\left\langle \frac{\eta ^{\ast }}{k}\right\rangle ^{-7}\left( -9\frac{\eta
^{\ast }}{k}+6\left( \frac{\eta ^{\ast }}{k}\right) ^{3}\right) \\
&=&\frac{N}{k^{3}}\frac{16}{27}\sqrt{3}>0
\end{eqnarray*}%
Thus near $\eta ^{\ast }$, we have
\begin{equation}
g\left( \eta \right) =g(\eta ^{\ast })+\frac{1}{6}g^{\prime \prime \prime
}\left( \eta ^{\ast }\right) \left( \eta -\eta ^{\ast }\right) ^{3}+o\left(
\left( \eta -\eta ^{\ast }\right) ^{3}\right) ,  \label{expansion-g}
\end{equation}%
and
\begin{equation}
g^{\prime }\left( \eta \right) =\frac{1}{2}g^{\prime \prime \prime }\left(
\eta ^{\ast }\right) \left( \eta -\eta ^{\ast }\right) ^{2}+o\left( \left(
\eta -\eta ^{\ast }\right) ^{2}\right) .  \label{expansion-g'}
\end{equation}%
Choose $\delta >0$ small such that (\ref{expansion-g}) and (\ref%
{expansion-g'}) hold true in $I=\left( \eta ^{\ast }-\delta ,\eta ^{\ast
}+\delta \right) $. In particular, $g^{\prime }\left( \eta \right) >0\ $when
$\eta \in I$ and $\eta \neq \eta ^{\ast }$, thus $g\left( \eta \right) $ is
monotone in $I$. For a function $f$ with its support in $I$, letting $%
u=g(\eta )$ we have
\begin{align*}
\hat{\psi}(t;k,ct)& =\frac{1}{\sqrt{2\pi }}\int_{-\infty }^{\infty }f(\eta
)e^{i\lambda t+i\eta ct}\mathrm{d}\eta \\
& =\frac{1}{\sqrt{2\pi }}\int_{-\infty }^{\infty }f(g^{-1}(u))e^{iut}\frac{1}{%
g^{\prime }\left( g^{-1}(u)\right) }\mathrm{d}u.
\end{align*}%
In the above, $\frac{1}{g^{\prime }\left( g^{-1}(u)\right) }$ has
singularity at $u^{\ast }=g(\eta ^{\ast })=\frac{4\sqrt{2}}{3\sqrt{3}}N$.
Since
\begin{equation*}
u=g\left( \eta \right) =u^{\ast }+O(\eta -\eta ^{\ast })^{3},\ \eta \in I,
\end{equation*}
so the order of singularity is
\begin{equation}
\frac{1}{g^{\prime }\left( g^{-1}(u)\right) }=O\left( \frac{1}{\left\vert
\eta -\eta ^{\ast }\right\vert ^{2}}\right) =O\left( \frac{1}{\left\vert
u-u^{\ast }\right\vert ^{\frac{2}{3}}}\right) .
\label{asymptotics-g'-inverse}
\end{equation}%
Choose
\begin{equation*}
f(\eta )=
\frac{g^{\prime }\left( \eta \right) }{\left\vert g\left( \eta
\right) -u^{\ast }\right\vert ^{\frac{2}{3}}}\chi _{I}\left( \eta \right) =%
\frac{g^{\prime }\left( g^{-1}(u) \right) }{\left\vert u-u^{\ast }\right\vert ^{\frac{2}{3}}}\chi _{I}\left( \eta \right) ,
\end{equation*}%
which by (\ref{asymptotics-g'-inverse}) is smooth in its support $I$. Hence
the inverse Fourier transform of $f\ $is smooth, and has finite $H_{y}^{s}$
norm for arbitrarily $s>0$. By (\ref{expansion-g}),
\begin{equation*}
a_{-}=g\left( \eta ^{\ast }-\delta \right) -g\left( \eta ^{\ast }\right)
<0,\ a_{+}=g\left( \eta ^{\ast }+\delta \right) -g\left( \eta ^{\ast
}\right) >0.
\end{equation*}%
Therefore, we have
\begin{align*}
\hat{\psi}(t;k,ct)& =\frac{1}{2\pi }\int_{g\left( \eta ^{\ast }-\delta
\right) }^{g\left( \eta ^{\ast }+\delta \right) }\frac{1}{\left\vert
u-u^{\ast }\right\vert ^{\frac{2}{3}}}e^{iut}\mathrm{d}u \\
& =\frac{e^{iu^{\ast }t}}{2\pi }\int_{a_{-}}^{a_{+}}\xi ^{-\frac{2}{3}%
}e^{i\xi t}\mathrm{d}\xi \\
& =\frac{e^{iu^{\ast }t}}{2\pi t^{\frac{1}{3}}}\int_{a_{-}t}^{a_{+}t}\left.
\xi ^{\prime }\right. ^{-\frac{2}{3}}e^{i\xi ^{\prime }}\mathrm{d}\xi
^{\prime },
\end{align*}%
while
\begin{equation*}
\lim_{t\rightarrow +\infty }\int_{a_{-}t}^{a_{+}t}\left. \xi ^{\prime
}\right. ^{-\frac{2}{3}}e^{i\xi ^{\prime }}\mathrm{d}\xi ^{\prime
}=\int_{-\infty }^{\infty }x^{-\frac{2}{3}}e^{ix}\mathrm{d}x=\sqrt{3}\Gamma
\left( \frac{1}{3}\right) .
\end{equation*}%
Therefore, $\Vert \hat{\psi}(t;k,\cdot )\Vert _{L_{y}^{\infty }}$ cannot
decay faster than $t^{-\frac{1}{3}}$.

\begin{remark}
The optimal $t^{-\frac{1}{3}}$ decay obatined above for $\left( x,y\right)
\in \mathbb{T\times R}$ is essentially for the one dimensional case (in $y$%
). By contrast, in \cite{elgindi-sima} the dispersive decay of solutions of (%
\ref{eqn-RT-B-1})-(\ref{eqn-RT-B-2}) was shown to be $t^{-\frac{1}{2}}$ for
the 2D case, i.e., $\left( x,y\right) \in \mathbb{R}^{2}$. The decay rate in
\cite{elgindi-sima} was obtained by the Littlewood-Paley decomposition and
stationary phase lemma.
\end{remark}

\subsection{Original Euler Equation}

When there is no shear, i.e. $R=0$, the original Euler equations (\ref%
{originalR01}-\ref{originalR02}) become
\begin{align*}
-\beta \partial _{t}\partial _{y}\psi +\partial _{t}\Delta \psi =& -\partial
_{x}\left( \frac{\rho }{\rho _{0}}\right) g, \\
\partial _{t}\left( \frac{\rho }{\rho _{0}}\right) =& \beta \partial
_{x}\psi .
\end{align*}%
Likewise, define $T=\frac{\rho }{\beta \rho _{0}(y)}$, then the equations
read
\begin{equation}
(-\beta \partial _{y}+\Delta )\psi _{t}=-\partial _{x}T\beta g,
\label{eqn-RT-original-1}
\end{equation}%
\begin{equation}
\partial _{t}T=\partial _{x}\psi .  \label{eqn-RT-original-2}
\end{equation}

Let $\Psi =e^{-\frac{1}{2}\beta y}\psi ,\Upsilon =e^{-\frac{1}{2}\beta y}T$,
then the equations (\ref{eqn-RT-original-1})-(\ref{eqn-RT-original-2})
become
\begin{equation}
\left( -\frac{1}{4}\beta ^{2}+\Delta \right) \Psi _{t}=-N^{2}\partial
_{x}\Upsilon ,\ \ \partial _{t}\Upsilon =\partial _{x}\Psi .
\label{eqn-RT-original-weighted}
\end{equation}%
By the Fourier transform $(x,y)\rightarrow (k,\eta )$, we have
\begin{equation*}
\left( -\frac{1}{4}\beta ^{2}+(i\eta )^{2}+(ik)^{2}\right) \hat{\Psi}%
_{t}=-(ik)N^{2}\hat{\Upsilon},\ \ \ \hat{\Upsilon}_{t}=(ik)\hat{\Psi}.
\end{equation*}%
%
%
%
%
Therefore,
\begin{equation*}
\frac{d^{2}}{dt^{2}}\hat{\Psi}=-\lambda ^{2}\hat{\Psi},
\end{equation*}%
where
\begin{equation*}
\lambda ^{2}=\frac{k^{2}N^{2}}{k^{2}+\eta ^{2}+\frac{\beta ^{2}}{4}}.
\end{equation*}%
Its solutions are
\begin{equation*}
\hat{\Psi}(t)=C_{1}e^{i\lambda t}+C_{2}e^{-i\lambda t},
\end{equation*}%
where
\begin{equation*}
C_{1,2}=\frac{1}{2}\left( \hat{\Psi}^{0}\pm \frac{\lambda }{k}\hat{\Upsilon}%
^{0}\right) .
\end{equation*}%
Similar to the Boussinesq case, we have the following conservation law for (%
\ref{eqn-RT-original-weighted})%
\begin{equation*}
0=\frac{\mathrm{d}}{\mathrm{d}t}\left( \iint \left( \frac{1}{4}\beta
^{2}\left\vert \Psi \right\vert ^{2}+\left\vert \nabla \Psi \right\vert
^{2}+N^{2}\left\vert \Upsilon \right\vert ^{2}\right) \mathrm{d}x\mathrm{d}%
y\right) .
\end{equation*}%
By integration by parts,
\begin{align*}
& \iint \left( \frac{1}{4}\beta ^{2}\left\vert \Psi \right\vert
^{2}+\left\vert \nabla \Psi \right\vert ^{2}+N^{2}\left\vert \Upsilon
\right\vert ^{2}\right) \mathrm{d}x\mathrm{d}y \\
=& \left\Vert e^{-\frac{1}{2}\beta y}v^{x}\right\Vert
_{L^{2}}^{2}+\left\Vert e^{-\frac{1}{2}\beta y}v^{y}\right\Vert _{L^{2}}^{2}+%
\frac{g}{\beta }\left\Vert e^{-\frac{1}{2}\beta y}\frac{\rho }{\rho _{0}}%
\right\Vert _{L^{2}}^{2}.
\end{align*}%
This shows that there is no decay in the $L^{2}$ norm for $e^{-\frac{1}{2}%
\beta y}\boldsymbol{v}$ and $e^{-\frac{1}{2}\beta y}\frac{\rho }{\rho _{0}}$. For the $%
L^{\infty }$ decay, notice that
\begin{equation*}
\lambda ^{2}=\frac{k^{2}N^{2}}{k^{2}+\eta ^{2}+\frac{\beta ^{2}}{4}}=\frac{%
m^{2}\left( \kappa N\right) ^{2}}{m^{2}+\eta ^{2}}.
\end{equation*}%
where $m=\sqrt{\frac{1}{4}\beta ^{2}+k^{2}},\kappa =\frac{k}{m}$. By lemma %
\ref{oscillatingIntegral} we have
\begin{align*}
\left\vert \int_{-n}^{n}e^{i(\lambda t+\eta y)}\mathrm{d}\eta \right\vert
\lesssim & |m|^{\frac{3}{2}}|\kappa Nt|^{-\frac{1}{3}}+|\kappa Nt|^{-\frac{1%
}{2}}|m|^{-\frac{1}{2}}n^{\frac{3}{2}} \\
\simeq & |k|^{\frac{3}{2}}|t|^{-\frac{1}{3}}+|t|^{-\frac{1}{2}}|k|^{-\frac{1%
}{2}}n^{\frac{3}{2}},
\end{align*}%
since $\kappa \simeq 1,m\simeq k$. Accordingly, we have
\begin{align*}
\Vert e^{-\frac{1}{2}\beta y}P_{\neq 0}\psi \Vert _{L_{x}^{2}L_{y}^{\infty
}}\lesssim & |t|^{-\frac{1}{3}}\left( \Vert \Psi ^{0}\Vert
_{H_{x}^{3/2}H_{y}^{7/2}}+\Vert \Psi ^{0}\Vert _{H_{x}^{3/2}L_{y}^{1}}\right.
\\
& \left. +\Vert \Upsilon ^{0}\Vert _{H_{x}^{1/2}H_{y}^{7/2}}+\Vert \Upsilon
^{0}\Vert _{H_{x}^{1/2}L_{y}^{1}}\right) , \\
\Vert e^{-\frac{1}{2}\beta y}P_{\neq 0}v^{x}\Vert _{L_{x}^{2}L_{y}^{\infty
}}\lesssim & |t|^{-\frac{1}{3}}\left( \Vert \Psi ^{0}\Vert
_{H_{x}^{3/2}H_{y}^{9/2}}+\Vert \Psi ^{0}\Vert
_{H_{x}^{3/2}W_{y}^{1,1}}\right. \\
& \left. +\Vert \Upsilon ^{0}\Vert _{H_{x}^{1/2}H_{y}^{9/2}}+\Vert \Upsilon
^{0}\Vert _{H_{x}^{1/2}W_{y}^{1,1}}\right) , \\
\Vert e^{-\frac{1}{2}\beta y}v^{y}\Vert _{L_{x}^{2}L_{y}^{\infty }}\lesssim
& |t|^{-\frac{1}{3}}\left( \Vert \Psi ^{0}\Vert
_{H_{x}^{5/2}H_{y}^{7/2}}+\Vert \Psi ^{0}\Vert _{H_{x}^{5/2}L_{y}^{1}}\right.
\\
& \left. +\Vert \Upsilon ^{0}\Vert _{H_{x}^{3/2}H_{y}^{7/2}}+\Vert \Upsilon
^{0}\Vert _{H_{x}^{3/2}L_{y}^{1}}\right) , \\
\Vert e^{-\frac{1}{2}\beta y}P_{\neq 0}T\Vert _{L_{x}^{2}L_{y}^{\infty
}}\lesssim & |t|^{-\frac{1}{3}}\left( \Vert \Psi ^{0}\Vert
_{H_{x}^{5/2}H_{y}^{9/2}}+\Vert \Psi ^{0}\Vert
_{H_{x}^{5/2}W_{y}^{1,1}}\right. \\
& \left. +\Vert \Upsilon ^{0}\Vert _{H_{x}^{3/2}H_{y}^{7/2}}+\Vert \Upsilon
^{0}\Vert _{H_{x}^{3/2}L_{y}^{1}}\right) .
\end{align*}

\begin{center}
{\Large Acknowledgement}
\end{center}

Yang is supported in part by China Scholarship Council. He would like to
thank Dongsheng Li for helpful suggestions. Lin is supported in part by a
NSF grant DMS-1411803.


\bibliographystyle{abbrv}
\bibliography{reg}

\end{document}